\documentclass{amsart}

\usepackage{xypic}
\input xy
\xyoption{all}
\usepackage{epsfig}
\usepackage{amsthm}
\usepackage{amssymb}
\usepackage{amsmath}
\usepackage{amscd}
\usepackage{color}
\usepackage[T1]{fontenc}
\usepackage{enumerate}

\newcommand{\bg}{\begin{equation}}
\newcommand{\ed}{\end{equation}}
\newcommand{\bga}{\begin{eqnarray}}
\newcommand{\eda}{\end{eqnarray}}
\newcommand{\pf}{\textbf{Proof:\ }}

\def\cbdu{\par{\raggedleft$\Box$\par}}

\newtheorem {Theorem}  {Theorem}

\numberwithin{Theorem}{section}

\newtheorem {Lemma}[Theorem]  {Lemma}
\newtheorem {Proposition}[Theorem]{Proposition}
\theoremstyle{definition}
\newtheorem{Definition}[Theorem]{Definition}
\theoremstyle{remark}
\newtheorem{Remark}[Theorem]{\bf Remark}

\expandafter\chardef\csname pre amssym.def
at\endcsname=\the\catcode`\@ \catcode`\@=11
\def\undefine#1{\let#1\undefined}
\def\newsymbol#1#2#3#4#5{\let\next@\relax
 \ifnum#2=\@ne\let\next@\msafam@\else
 \ifnum#2=\tw@\let\next@\msbfam@\fi\fi
 \mathchardef#1="#3\next@#4#5}
\def\mathhexbox@#1#2#3{\relax
 \ifmmode\mathpalette{}{\m@th\mathchar"#1#2#3}%
 \else\leavevmode\hbox{$\m@th\mathchar"#1#2#3$}\fi}
\def\hexnumber@#1{\ifcase#1 0\or 1\or 2\or 3\or 4\or 5\or 6\or 7\or 8\or
 9\or A\or B\or C\or D\or E\or F\fi}

\font\teneufm=eufm10 \font\seveneufm=eufm7 \font\fiveeufm=eufm5
\newfam\eufmfam
\textfont\eufmfam=\teneufm \scriptfont\eufmfam=\seveneufm
\scriptscriptfont\eufmfam=\fiveeufm

\catcode`\@=\csname pre amssym.def at\endcsname

\newcounter{remark}
\newcommand{\R}{\mathbf{R}}

\renewcommand{\div}{\mbox{div}}

\def  \R   {{\mathbb R}}

\def  \12  {{\frac{1}{2}}}

\def\build#1_#2^#3{\mathrel{\mathop{\kern 0pt#1}\limits_{#2}^{#3}}}

\begin{document}

\title[Regularity criteria for NSE and MHD]{Regularity criteria for the 3D Navier-Stokes and MHD
equations}

\author [Alexey Cheskidov]{Alexey Cheskidov}
\address{Department of Mathematics, Stat. and Comp. Sci.,  University of Illinois Chicago, Chicago, IL 60607,USA}
\email{acheskid@uic.edu} 
\author [Mimi Dai]{Mimi Dai}
\address{Department of Mathematics, Stat. and Comp. Sci.,  University of Illinois Chicago, Chicago, IL 60607,USA}
\email{mdai@uic.edu} 

\thanks{The work of A. Cheskidov was partially supported by NSF Grant
DMS--1517583 and  DMS--1909849; the work of M. Dai was partially supported by NSF Grants DMS--1815069 and DMS-- 2009422.}

\begin{abstract}
We prove that a solution to the 3D Navier-Stokes or MHD equations does not blow up at $t=T$ provided
$\displaystyle \limsup_{q \to \infty} \int_{\mathcal{T}_q}^T \|\Delta_q(\nabla \times u)\|_\infty \, dt$ is small enough,
where $u$ is the velocity, $\Delta_q$ is the Littlewood-Paley projection, and $\mathcal T_q$ is a certain sequence
such that $\mathcal T_q \to T$ as $q \to \infty$. This improves many existing regularity criteria.

\bigskip


\hspace{0.02cm}CLASSIFICATION CODE: 76D03, 35Q35.
\end{abstract}

\maketitle

\section{Introduction}

The three-dimensional incompressible
magneto-hydrodynamics (MHD) equations are given by
\begin{equation}\label{MHD}
\begin{split}
u_t+u\cdot\nabla u-b\cdot\nabla b+\nabla p=\nu\Delta u,\\
b_t+u\cdot\nabla b-b\cdot\nabla u=\mu\Delta b,\\
\nabla \cdot u=0, \ \ \ \nabla \cdot b=0.
\end{split}
\end{equation}
Here $x\in\mathbb{R}^3$, $t\geq 0$, $u$ is the fluid velocity, $p$ is the fluid pressure and $b$
is the magnetic field. The parameter $\nu$ denotes the kinematic viscosity coefficient of the fluid, and $\mu$ is the reciprocal of the magnetic Reynolds number. 
The MHD equations are supplemented with initial conditions
\begin{equation}
u(x,0)=u_0(x),\;\;\; b(x,0)=b_0(x),
\end{equation}
where $u_0$ and $b_0$ are divergence free vector fields in $L^2(\mathbb{R}^3)$.
When the magnetic field $b(x,t)$ vanishes,
(\ref{MHD}) reduce to the incompressible Navier-Stokes
equations (NSE). Solutions to the MHD equations also share the same scaling
property with solutions to the NSE, that is,
\begin{equation}
u_\lambda(x,t) =\lambda u(\lambda x, \lambda^2t), \ b_\lambda(x,t)
=\lambda b(\lambda x, \lambda^2t), \ p_\lambda(x,t) =\lambda^2
p(\lambda x, \lambda^2t) \notag
\end{equation}
solve (\ref{MHD}) with the initial data
\begin{equation}
u_{0\lambda} =\lambda u_0(\lambda x), \ b_{0\lambda} =\lambda
b_0(\lambda x), \notag
\end{equation}
provided $u(x, t)$ and $b(x, t)$ solve (\ref{MHD}) with the initial data $u_0(x)$ and $b_0(x)$. Spaces that are invariant under the above scaling are called critical.
Examples of critical spaces include the Prodi-Serrin spaces
\[
L^l(0,\infty; L^r), \qquad \frac{2}{l}+\frac{3}{r}=1,
\]
or their Besov counterparts
\[
L^l\left(0,\infty; B^{-1+\frac{2}{l}+\frac{3}{r} }_{r,\infty}\right).
\]

The existence of global regular solutions for the NSE and the MHD system in 3D is an outstanding open problem. The study of these models in critical spaces has been one of the focuses 
of research activities since the initial work of Kato \cite{Kato}. It is known that if a solution belongs  to a certain critical space, then it is regular.
For instance, a smooth solution to the  NSE $(0,T)$ does not blow up at $T$ if
\begin{equation}\label{BKMcriterion}
\int_0^T\|\nabla \times u\|_\infty\, dt<\infty,
\end{equation}
which is known as the Beale-Kato-Majda (BKM) condition \cite{BKM}. In term of scaling, this is the borderline of Prodi-Serrin criteria when the time integrability index is one, which does not
require any dissipation and also applies to the Euler equations.

For the Euler equations, the BKM condition \eqref{BKMcriterion} was weakened by Planchon \cite{P} to
\begin{equation}\label{BKM-extended}
\lim_{\varepsilon\to 0}\sup_{q}\int_{T-\varepsilon}^T\|\Delta_q(\nabla\times u)\|_{\infty}\, dt<c,
\end{equation}
for small enough $c$, where $\Delta_q$ is the Littlewood-Paley projection.

For the 3D NSE, Cheskidov and Shyvdkoy \cite{CS} proved the following improvement of the Beale-Kato-Majda condition \eqref{BKMcriterion}:
\begin{equation} \label{eq:BKM-Lambda}
\int_0^T\|\nabla \times u_{\leq Q}\|_{B^0_{\infty,\infty}}\, dt<\infty, 
\end{equation}
where $2^{Q(t)}=\Lambda(t)$ and $\Lambda(t)$ is a wavenumber constructed based on the structure of the NSE.
This condition is also weaker than all the Prodi-Serrin type regularity criteria. Indeed,
\[
u_{\leq Q} \in L^r(0,T; B^{-1+\frac{2}{r}}_{\infty,\infty}) \qquad \text{for some} \qquad 1\leq r <\infty
\]
implies \eqref{eq:BKM-Lambda} (see \cite{CS}). Note that $r$ cannot be taken as $\infty$ here. A regularity criterion in the limit case $r=\infty$ has been established by Escauriaza, Seregin, and Sverak \cite{ESS} in $L^\infty(0;T,L^3)$, and
extended by Gallagher, Koch, and Planchon to $L^\infty(0,T;B^{-1+ \frac{3}{s}}_{s,q})$ with $3 < s,q< \infty$ in \cite{GKP}.

For the MHD system, Caflisch, Klapper and Steele \cite{CKS} obtained the BKM-type condition
in the inviscid case, but in terms of both $u$ and $b$. Cannone, Chen, and Miao \cite{CCM} 
proved the extended version of the BKM criterion, as \eqref{BKM-extended}, but also in terms of
both $u$ and $b$.
In the viscous case, Wu \cite{Wu04} obtained some Prodi-Serrin type criteria which impose conditions on both $u$ and $b$. However there have been various indications, both numerically \cite{Ha, PPS} and theoretically,  
suggesting that the velocity field $u$ plays a dominant role in interactions between $u$ and $b$. This point is still not clear though. In fact, in \cite{DQS}, the authors showed that even when the initial velocity vanishes, one can observe norm inflation in velocity at a small time due to the interaction of the magnetic field and the velocity field (see also \cite{CD} for norm inflation results in a
wider range of spaces). Nevertheless, a large amount of work has been devoted to establishing criteria which impose 
conditions only on the velocity, see \cite{CG, CMZ08, CMZ, HW, HX, LZ, Wu08, Zh05}.  In particular, He, Xin, and Zhou \cite{HX,Zh05} obtained the following Prodi-Serrin conditions for regularity:
\begin{equation}\label{PS}
\begin{split}
&u\in L^l(0,T; L^r) \qquad \mbox { with } \qquad \frac2l+\frac3r=1 \qquad \mbox { for } 3<r\leq \infty;\\
&\nabla u\in L^l(0,T; L^r) \qquad \mbox { with } \qquad \frac2l+\frac3r=2 \qquad \mbox { for } 3< r\leq \infty.
\end{split}
\end{equation}
He and Xin \cite{HX} also obtained the following regularity condition:
\begin{equation} \label{eq:ContCondition}
u\in C([0,T]; L^3).
\end{equation}
Finally, remarkable improvements were made by Chen, Miao and Zhang who obtained the extended BKM criterion \eqref{BKM-extended} in
terms of the velocity $u$ only \cite{CMZ},
as well as the Prodi-Serrin regularity criterion in terms of Besov spaces \cite{CMZ08}:
\begin{equation}\label{CMZ08}
\begin{split}
u\in L^l(0,T;B^s_{r,\infty}) \qquad \mbox { with } \qquad \frac2l+\frac3r=1+s, \\
\frac3{1+s}<r\leq \infty, -1<s\leq 1, \qquad
 \mbox { and } \qquad (r,s)\neq(\infty, 1).
 \end{split}
\end{equation}

In this paper, we will establish a new regularity criterion for the MHD system, which imposes condition only on low modes of the velocity field.
The criterion resembles the extended Beale-Kato-Majda condition \eqref{BKM-extended}  \cite{CCM, CMZ,P}, but in fact it is weaker as we take
advantage of a dissipation wavenumber $\Lambda_r(t)$ for $r\in[2,6)$ similar to the one introduced by Cheskidov and Shvydkoy \cite{CS} for the 3D NSE.  
The dissipation wavenumber $\Lambda_r(t)$ is defined as
\begin{equation}\label{wave}
\Lambda_r(t)=\min \left\{\lambda_q:\lambda_p^{-1+\frac 3r}\|u_p(t)\|_r<c_r\min\{\nu,\mu\}, \forall p>q, q\in \mathbb{N} \right\},
\end{equation}
where $\lambda_q =2^q$, $u_p = \Delta_p u$ is the Littlewood-Paley projection of $u$,  and $c_r$ is an adimensional  constant that depends only on $r$. Here $r\in [2,6)$ in the case of the
the MHD system (c.f. proof of Proposition \ref{prop}), and $r \in [2,\infty]$ for the NSE (when $b \equiv 0$). Let $Q_r(t)\in\mathbb N$ be such that $\lambda_{Q_r(t)}=\Lambda_r(t)$.

The wavenumber $\Lambda_r(t)$, defined for each individual solution (in terms of $u$ only), separates the inertial range from the dissipation range where
the viscous terms $\Delta u$ and $\Delta b$ dominate.
The dissipation wavenumber is
defined so that the critical $B^{-1+\frac{3}{r}}_{\infty,\infty}$ norm of $u$ is small above $\Lambda_r$ for some $2\leq r<6$. So it is larger than the
dissipation wavenumber for the NSE defined with $r=\infty$ in \cite{CS}. This is due to the fact that there are less cancellations and more terms
to control  for the MHD system.
It is worth mentioning that  this method is also applied to the supercritical quasi-geostrophic equation in \cite{Dsqg} and to the Hall-magneto-hydrodynamics system in \cite{Dhmhd}.
Regularity criteria obtained in these papers are weaker than all the corresponding Prodi-Serrin type criteria.  
The dissipation wavenumber is also related to the determining wavenumber used in \cite{CD-sqg-determ, CDK} to estimate the number of determining modes for
fluid flows.

Our main result states as follows.
\begin{Theorem}\label{thm}
Let $(u, b)$ be a weak solution to (\ref{MHD}) on $[0,T]$. Assume that $(u(t), b(t))$ is regular on $(0,T)$. If 
for any $2\leq r<6$ ($2\leq r \leq \infty$ when $b\equiv 0$),
\bg\label{criterion}
\limsup_{q \to \infty} \int_{T/2}^T 1_{q\leq Q_r(t)}\|\Delta_q(\nabla \times u)\|_\infty \, dt  \leq c_r,
\ed
with a small constant $c_r>0$ depending on $r$, then $(u(t), b(t))$ is regular on $(0,T]$.
\end{Theorem}

\begin{Remark}
The choice of constant $c_r$ will be made specific later on, see (\ref{parameter-cr}).
\end{Remark}

We will show that \eqref{criterion} is weaker than various existing regularity criteria. Indeed,
define
\[
\mathcal{T}_q = \sup\{t \in (T/2,T): Q_r(\tau) < q,  \ \forall \tau \in(T/2,t)\}.
\]
Note that $\mathcal{T}_q$ is increasing. In addition, we show that
\[
\lim_{q \to \infty} \mathcal{T}_q = T.
\]
\begin{Theorem} \label{thm:compare}
Regularity condition \eqref{criterion} is satisfied provided one of the following holds.
\begin{enumerate}
\item
$\displaystyle \limsup_{q \to \infty} \int_{\mathcal{T}_q}^T \|\Delta_q(\nabla \times u)\|_\infty \, dt \leq c_r.$
\vspace{0.1in}
\item
$\displaystyle\lim_{\varepsilon \to 0} \limsup_{q \to \infty} \int_{T-\varepsilon}^T\|\Delta_q(\nabla \times u)\|_\infty \, dt \leq c_r.$
\vspace{0.1in}
\item
$\displaystyle\int_0^T \sup_{q \leq Q_r(t)} \|\Delta_q(\nabla \times u)\|_\infty \, dt < \infty$.
\vspace{0.1in}
\item
$\displaystyle  \limsup_{q \to \infty}\int_0^T 1_{q\leq Q_r(t)}\left(\lambda_q^{-1+\frac2l+\frac{3}{r}}\|u_q(t)\|_{r}\right)^l\, dt \lesssim c_r^l \min\{\nu,\mu\}^{l-1}$,  $1\leq l< \infty$, $2\leq r<6$ ($2\leq r\leq \infty$ when $b\equiv 0$).
\vspace{0.1in}
\item
$\displaystyle \limsup_{q \to \infty}\int_{\mathcal{T}_q}^T \left(\lambda_q^{-1+\frac2l+\frac{3}{r}}\|u_q(t)\|_{r}\right)^l\, dt \lesssim c_r^l\min\{\nu,\mu\}^{l-1}$,  $1\leq l<\infty$, $2\leq r<6$ ($2\leq r\leq \infty$ when $b\equiv 0$).
\vspace{0.1in}
\item
$\displaystyle \lim_{\epsilon \to 0} \limsup_{q \to \infty} \int_{T-\epsilon}^T \left(\lambda_q^{-1+\frac2l+\frac{3}{r}}\|u_q(t)\|_{r}\right)^l\, dt \lesssim c_r^l\min\{\nu,\mu\}^{l-1}$,  $1\leq l<\infty$, $2\leq r<6$ ($2\leq r\leq \infty$ when $b\equiv 0$).
\vspace{0.1in}
\item
$\displaystyle \int_0^T \left( \|u_{\leq Q_r(t)}\|_{B^{-1+ \frac{2}{l} + \frac{3}{r}}_{r,\infty}}\right)^l \, dt
< \infty$,   $1\leq l<\infty$, $2\leq r<6$ ($2\leq r\leq \infty$ when $b\equiv 0$).
\vspace{0.1in}
\item
$\displaystyle \limsup_{t \to T-} \|u(t)-u(T)\|_{B^{-1+ \frac{3}{r}}_{r,\infty}} < \frac{c_r}{2}\min\{\nu,\mu\}$, $2\leq r<6$
($2\leq r\leq \infty$ when $b\equiv 0$).
\end{enumerate}
\end{Theorem}

\begin{Remark}
In the case of the 3D NSE ($b \equiv 0$), the value $r= \infty$ in both Theorem \ref{thm} and Theorem \ref{thm:compare} gives the best regularity criteria due to Bernstein's inequality.
\end{Remark}

\begin{Remark}
If $u\in C((0,T],L^3)$, then $Q_r \in L^\infty(T/2,T)$ for any $3\leq r \leq \infty$, and every regularity condition in Theorems~\ref{thm}, \ref{thm:compare} is automatically satisfied. However, these regularity conditions 
exclude the $L^\infty(0;T,L^3)$ criterion by Escauriaza, Seregin, and Sverak \cite{ESS}. Notably, recent  Tao's result for the 3D NSE \cite[Theorem 5.1]{Tao} implies that $Q_\infty \in L^\infty(T/2,T)$ provided $u\in L^\infty(0;T,L^3)$.
What is even more remarkable, Tao's result provides a triple exponential control of $\Lambda_\infty$ in terms of $\|u\|_3$, and gives quantitative (triple logarithmic) blow up rate for  $\|u\|_3$.
\end{Remark}

For the 3D NSE, conditions (6), (8), and (3), (7) were obtained by Cheskidov, Shvydkoy in \cite{CS10} and  \cite{CS} respectively.
Regarding the MHD system, condition (7) (and hence (4), (5), and (6)) is weaker than \eqref{PS}, \eqref{CMZ08} for $r<6$, and condition (8) is weaker than \eqref{eq:ContCondition}.
Conditions (1) and (2) are weaker than \eqref{BKM-extended} for both the 3D NSE and MHD system.

The rest of the paper is organized as follows. In Section \ref{sec:pre} we introduce some notations, recall the Littlewood-Paley theory and definitions of weak and regular solutions. Section \ref{sec:reg} is devoted to proving Theorem \ref{thm} and Theorem \ref{thm:compare}.

\bigskip

\section{Preliminaries}
\label{sec:pre}

\subsection{Notation}
\label{sec:notation}
We denote by $A\lesssim B$ an estimate of the form $A\leq C B$ with
some absolute constant $C$, and by $A\sim B$ an estimate of the form $C_1
B\leq A\leq C_2 B$ with some absolute constants $C_1$, $C_2$. 
We write $\|\cdot\|_p=\|\cdot\|_{L^p}$. The symbol $(\cdot, \cdot)$ stands for the $L^2$-inner product.

\subsection{Littlewood-Paley decomposition}
\label{sec:LPD}
Our method  relies on the Littlewood-Paley decomposition, which we briefly recall here. We refer the readers to the books by Bahouri, Chemin and Danchin \cite{BCD} and Grafakos \cite{Gr} for a more detailed description on this theory.

Let $\mathcal F$ and $\mathcal F^{-1}$ denote the Fourier transform and inverse Fourier transform, respectively.  A nonnegative radial function $\chi\in C_0^\infty(\R^n)$ is chosen such that 
\begin{equation}\label{def_of_chi}
\chi(\xi)=
\begin{cases}
1, \ \ \mbox { for } |\xi|\leq\frac{3}{4}\\
0, \ \ \mbox { for } |\xi|\geq 1.
\end{cases}
\end{equation}
Let 
\bg\notag
\varphi(\xi)=\chi(\xi/2)-\chi(\xi),
\ed
and
\begin{equation}\notag
\varphi_q(\xi)=
\begin{cases}
\varphi(\lambda_q^{-1}\xi)  \ \ \ \mbox { for } q\geq 0,\\
\chi(\xi) \ \ \ \mbox { for } q=-1.
\end{cases}
\end{equation}
Note that the sequence of the smooth functions $\varphi_q$ forms a dyadic partition of unity.
For a tempered distribution vector field $u$ we define the Littlewood-Paley decomposition
\begin{equation}\notag
\begin{cases}
&h=\mathcal F^{-1}\varphi, \qquad \tilde h=\mathcal F^{-1}\chi,\\
&u_q:=\Delta_qu=\mathcal F^{-1}(\varphi(\lambda_q^{-1}\xi)\mathcal Fu)=\lambda_q^n\int h(\lambda_qy)u(x-y)dy,  \qquad \mbox { for }  q\geq 0,\\
& u_{-1}=\mathcal F^{-1}(\chi(\xi)\mathcal Fu)=\int \tilde h(y)u(x-y)dy.
\end{cases}
\end{equation}
Then
\bg\notag
u=\sum_{q=-1}^\infty u_q
\ed
holds in the distributional sense.  To simplify the notation, we denote
\bg\notag
\tilde u_q=u_{q-1}+u_q+u_{q+1}, \qquad u_{\leq Q}=\sum_{q=-1}^Qu_q,  \qquad u_{(P,Q]}=\sum_{q=P+1}^Qu_q.
\ed

The Besov space $B_{p,\infty}^{s}$ is defined as follows.
\begin{Definition}
Let $s\in \mathbb R$, and $1\leq p\leq \infty$. Let
$$
\|u\|_{B_{p, \infty}^{s}}=\sup_{q\geq -1}\lambda_q^s\|u_q\|_p.
$$
The Besov space $B_{p,\infty}^{s}$ is the space of tempered distributions $u$ such that  $\|u\|_{B_{p,\infty}^{s}}$ is finite.
\end{Definition}
Note that 
\[
  \|u\|_{H^s} \sim \left(\sum_{q=-1}^\infty\lambda_q^{2s}\|u_q\|_2^2\right)^{1/2},
\]
for each $u \in H^s$ and $s\in\R$.

We recall Bernstein's inequality as follows.
\begin{Lemma}\label{le:bern} \cite{L}
Let $n$ be the space dimension and $r\geq s\geq 1$. Then for all tempered distributions $u$, 
\bg\notag
\|u_q\|_{r}\lesssim \lambda_q^{n(\frac{1}{s}-\frac{1}{r})}\|u_q\|_{s}.
\ed
\end{Lemma}
Throughout the paper we will utilize the following sharp version of Bony's paraproduct decomposition 
\begin{equation}\notag
\begin{split}
\Delta_q(u\cdot\nabla v)=&\sum_{|q-p|\leq 2}\Delta_q(u_{\leq{p-2}}\cdot\nabla) v_p+
\sum_{|q-p|\leq 2}\Delta_q(u_{p}\cdot\nabla) v_{\leq{p-2}}\\
&+\sum_{p\geq q-2}\Delta_q(u_p\cdot\nabla)\tilde v_p,
\end{split}
\end{equation}
which holds because we insist on a rather specific choice of the bump function $\varphi$ via \eqref{def_of_chi}:
\[
\varphi(\xi)=
\begin{cases}
1, \ \ \mbox { for } 1\leq|\xi|\leq 3/2\\
0, \ \ \mbox { for } |\xi|\leq 3/4 \ \text{or} \ |\xi|\geq 2,
\end{cases}
\]
see Section 2.3 of \cite{CD-attractor} for more details.

We will also employ the commutator notation
\begin{equation}\notag
[\Delta_q, u_{\leq{p-2}}\cdot\nabla]v_p=\Delta_q(u_{\leq{p-2}}\cdot\nabla) v_p-(u_{\leq{p-2}}\cdot\nabla) \Delta_qv_p,
\end{equation}
which allows to essentially move the derivative from $v_p$ to $u_{\leq p-2}$.

\begin{Lemma}\label{le-comm}
For any $1<r_1<\infty$ and $r_1\leq r_2, r_3\leq\infty$ with $\frac1{r_2}+\frac1{r_3}=\frac1{r_1}$, the commutator satisfies 
\begin{equation}\label{est-comm}
\|[\Delta_q,u_{\leq{p-2}}\cdot\nabla] v_p\|_{r_1}
\lesssim \|v_p\|_{r_3}\sum_{p' \leq p-2} \lambda_{p'} \|u_{p'}\|_{r_2}.
\end{equation}
\end{Lemma}

\pf
We rewrite the commutator by applying the definition of $\Delta_q$, integration by parts, first order Taylor's formula, and a change of variable
\begin{equation}\notag
\begin{split}
&[\Delta_q, u_{\leq{p-2}}\cdot\nabla]v_p\\
=&\ \int_{\mathbb R^3} \lambda_q^3h(\lambda_q(x-y))\left(u_{\leq p-2}(y)-u_{\leq p-2}(x)\right)\cdot \nabla_y v_p(y) dy\\
=&\ -\int_{\mathbb R^3} \lambda_q^3\nabla_yh(\lambda_q(x-y))\left(\left(u_{\leq p-2}(y)-u_{\leq p-2}(x)\right)\otimes v_p(y)\right) dy\\
=&\ -\int_0^1\int_{\mathbb R^3} \lambda_q^3\nabla_yh(\lambda_q(x-y))\left(\left((y-x)\cdot\nabla u_{\leq p-2}(x+\tau (y-x))\right)\otimes v_p(y)\right) dyd\tau\\
=&\ \int_0^1\int_{\mathbb R^3} \lambda_q^3\nabla_zh(\lambda_qz)\left(\left(z\cdot\nabla u_{\leq p-2}(x-\tau z)\right)\otimes v_p(x-z)\right) dzd\tau .\\
\end{split}
\end{equation}
To continue, we apply Minkowski's integral inequality, H\"older's inequality (with respect to $x$ variable), and the translation invariance of Lebesgue norm, and obtain
\begin{equation}\notag
\begin{split}
&\|[\Delta_q,u_{\leq{p-2}}\cdot\nabla] u_q\|_{r_1}\\
\lesssim & \int_0^1\int_{\mathbb R^3} \lambda_q^3|z|\left|\nabla_zh(\lambda_qz)\right| \left\|\nabla u_{\leq p-2}(\cdot-\tau z) v_p(\cdot-z)\right\|_{r_1} dzd\tau\\
\lesssim & \int_0^1\int_{\mathbb R^3} \lambda_q^3|z|\left|\nabla_zh(\lambda_qz)\right| \left\|\nabla u_{\leq p-2}(\cdot-\tau z)\right\|_{r_2} \left\|v_p(\cdot-z)\right\|_{r_3} dzd\tau\\
\lesssim &\|v_p\|_{r_3}\sum_{p' \leq p-2} \lambda_{p'} \|u_{p'}\|_{r_2} \int_0^1\int_{\mathbb R^3} \lambda_q^3|z|\left|\nabla_zh(\lambda_qz)\right| dzd\tau\\
\lesssim &\|v_p\|_{r_3}\sum_{p' \leq p-2} \lambda_{p'} \|u_{p'}\|_{r_2},
\end{split}
\end{equation}
where we used the fact that $\nabla h\in L^1$ in the last step. 

\cbdu

\bigskip

\subsection{Definition of solutions}
\label{sec:sol}
We recall some classical definitions of weak and regular solutions to a differential equation system (see, e.g.,  \cite{DL,ST, Tem}).
\begin{Definition}\label{def:weak} 
A weak solution of (\ref{MHD}) on $[0,T]$ is a pair of divergence free functions $(u, b)$ in the class 
$$
u,b \in C_w([0,T]; L^2(\mathbb R^3)) \cap L^2(0,T; H^1(\mathbb R^3)),
$$
satisfying 
\begin{equation}\notag
\begin{split}
&(u(t), \phi(t))-(u_0, \phi(0))\\
=&\int_0^t(u(s), \partial_s\phi(s))+\nu(u(s), \Delta\phi(s))+ (u(s)\cdot\nabla\phi(s), u(s))-(b(s)\cdot\nabla\phi(s), b(s))\, ds,
\end{split}
\end{equation}
\begin{equation}\notag
\begin{split}
&(b(t), \phi(t))-(b_0, \phi(0))\\
=&\int_0^t(b(s), \partial_s\phi(s))+\mu(b(s), \Delta\phi(s))+ (u(s)\cdot\nabla\phi(s), b(s))-(b(s)\cdot\nabla\phi(s), u(s))\, ds,
\end{split}
\end{equation}
for all test functions $\phi\in C_0^\infty([0,T]\times\mathbb R^3)$ with $\nabla_x\cdot \phi=0$.
\end{Definition}

It is known that space $\dot H^{\frac12}$ is critical for the 3D NSE and $\dot H^{\frac12}\times \dot H^{\frac12}$ is critical for the 3D MHD under the aforementioned natural scaling $u_\lambda(x,t)=\lambda u(\lambda x,\lambda^2 t)$ and $b_\lambda(x,t)=\lambda b(\lambda x,\lambda^2 t)$. The continuity of a subcritical norm on a time interval implies that the solution is smooth on this interval. Thus we adapt the following definition of a regular solution.

\begin{Definition}\label{le-reg} 
A weak solution $(u, b)$ of (\ref{MHD}) is regular on a time interval $\mathcal I$ if $\|u(t)\|_{H^s}$ and $\|b(t)\|_{H^s}$ are continuous on $\mathcal I$ for some $s> \frac{1}{2}$.
\end{Definition}

\begin{Remark} {\color{white}.....}\\
\begin{enumerate}
\item A weak solution $(u, b)$ of (\ref{MHD}) is regular on $(0,T)$ if and only if $(u,b) \in C^{\infty}(\mathbb{R}^3\times (0,T))$.
\item A weak solution $(u, b)$ of (\ref{MHD}) is regular on $(0,T]$ if and only if $(u,b) \in C^{\infty}(\mathbb{R}^3\times (0,T))$ and
\[
\limsup_{t \to T-} (\|u(t)\|_{H^s} + \|b(t)\|_{H^s}) < \infty,
\] 
for some $s>\frac{1}{2}$.
\item If a weak solution $(u, b)$ of (\ref{MHD}) is regular on $(0,T]$, then it can be extended as a smooth solution to a longer interval $(0,T + \varepsilon)$, for some $\varepsilon>0$.
\end{enumerate}
\end{Remark}

\bigskip

\section{Regularity Criterion}
\label{sec:reg}

In this section we will establish the regularity criterion in Theorem \ref{thm}. Let $(u(t), b(t))$ be a weak solution of (\ref{MHD}) on $[0,T]$.
Recall our definition of the dissipation wavenumber $\Lambda_r(t)$ as in (\ref{wave}). We will see that the constant $c_r$ in the definition needs to be chosen sufficiently small in the course of the proof of Theorem \ref{thm}. Recall $Q_r(t)\in\mathbb N$ be such that $\lambda_{Q_r(t)}=\Lambda_r(t)$. It follows immediately that
\[
\|u_p(t)\|_r < \lambda_p^{1-\frac{3}{r}} c_r\min\{\nu,\mu\}, \qquad \forall p>Q_r(t),
\]
and
\bg\label{Q}
\|u_{Q_r(t)}(t)\|_r\geq c_r\min\{\nu,\mu\}\Lambda_r^{1-\frac3r}(t),
\ed
provided $1<\Lambda_r(t)<\infty$. For simplicity, from now on we drop the subscript $r$ in $Q_r$ unless specified otherwise. We also denote
\[
f(t)=\sum_{q \leq Q(t)} \lambda_q \|u_{q}(t)\|_{\infty}.
\]

\subsection{Proof of Theorem \ref{thm}}
In order to prove that $(u,b)$ does not blow up at $T$, it is sufficient to show that $\|u(t)\|_{H^s}+\|b(t)\|_{H^s}$ is  bounded on $[T/2,T)$ for some $s>\frac{1}{2}$.  Before getting into lengthy computations and estimates, we first outline the scheme of the proof:\\
\begin{enumerate}[(i)]
\item Establish the main energy estimate stated in the following proposition. 
\begin{Proposition}\label{prop}
For any $2\leq r<6$ and $\frac12<s<\frac3r$ (any $2\leq r \leq \infty$ and $s>\frac12$ in case $b \equiv 0$), there exist adimensional constants $\tilde c$ and $C$, such that
\begin{equation}\label{energy-m}
\frac{d}{dt}\left(\|u\|_{H^s}^2+\|b_q\|_{H^s}^2\right)
\leq C f(t)\left(\|u_q\|_{H^s}^2+\|b_q\|_{H^s}^2\right),
\end{equation}
provided $c_r \leq \tilde c$.
\end{Proposition}
\item Show that one can weaken the condition $f\in L^1$ to the assumption (\ref{criterion}), which still allows us to apply Gr\"onwall's inequality to (\ref{energy-m}).
\item Complete the claim of Theorem \ref{thm}.\\
\end{enumerate}

\noindent {\bf Step (i):}
This step is devoted to the proof of Proposition \ref{prop}, which consists of heavy energy estimates through standard frequency localization techniques and the wavenumber splitting approach developed in the first author's previous work \cite{CS}. 

Since $(u(t), b(t))$ is regular on $(0,T)$, multiplying equations  (\ref{MHD}) with $\Delta^2_qu$ and $\Delta^2_qb$ respectively yields 
\begin{equation}\notag
\begin{split}
\frac{1}{2}\frac{d}{dt}\|u_q\|_2^2\leq &-\nu\|\nabla u_q\|_2^2-\int_{\R^3}\Delta_q(u\cdot\nabla u)\cdot u_q\, dx\\
&+\int_{\R^3}\Delta_q(b\cdot\nabla b)\cdot u_q\, dx,\\
\frac{1}{2}\frac{d}{dt}\|b_q\|_2^2\leq &-\mu\|\nabla b_q\|_2^2-\int_{\R^3}\Delta_q(u\cdot\nabla b)\cdot b_qdx\\
&+\int_{\R^3}\Delta_q(b\cdot\nabla u)\cdot b_qdx.
\end{split}
\end{equation}
Multiplying the above two inequalities by $\lambda_q^{2s}$, adding the resulted inequalities, and taking summation for all $q\geq -1$, 
we obain
\begin{equation}\label{ineq-ubq}
\begin{split}
\frac{1}{2}\frac{d}{dt}\sum_{q\geq -1}\lambda_q^{2s}\left(\|u_q\|_2^2+\|b_q\|_2^2\right)\leq &-\sum_{q\geq -1}\lambda_q^{2s}\left(\nu\|\nabla u_q\|_2^2+\mu\|\nabla b_q\|_2^2\right)\\
&-(I+J+K+L),
\end{split}
\end{equation}
with
\begin{equation}\notag
\begin{split}
I=&\sum_{q\geq -1}\lambda_q^{2s}\int_{\R^3}\Delta_q(u\cdot\nabla u)\cdot u_q\, dx, \qquad
J=-\sum_{q\geq -1}\lambda_q^{2s}\int_{\R^3}\Delta_q(b\cdot\nabla b)\cdot u_q\, dx,\\
K=&\sum_{q\geq -1}\lambda_q^{2s}\int_{\R^3}\Delta_q(u\cdot\nabla b)\cdot b_q\, dx,\qquad
L=-\sum_{q\geq -1}\lambda_q^{2s}\int_{\R^3}\Delta_q(b\cdot\nabla u)\cdot b_q\, dx.
\end{split}
\end{equation}
In the rest of this step, we will establish estimates for the flux terms $I, J, K,$ and $L$, as respectively stated in Lemma \ref{le-I},  Lemma \ref{le-J}, and Lemma \ref{le-K}. Proposition \ref{prop} is then an immediate consequence of (\ref{ineq-ubq}) and the estimates in Lemmas  \ref{le-I}-\ref{le-K}.

In order to estimate the flux terms $I, J, K,$ and $L$, we will employ extensively standard harmonic analysis techniques, Bony's paraproduct and commutator estimates. The key ingredient is that we split each flux term into high frequency and low frequency parts via the dissipation wavenumber $\Lambda_r(t)$, such that the high frequency part can be absorbed by the diffusion. 

\begin{Lemma}\label{le-I}
The flux term $I$ satisfies, for any $s>0$ and $r\geq 2$,
\begin{equation}\label{est-i}
|I|\lesssim c_r\nu\sum_{q> Q-3}\lambda_q^{2s+2}\|u_q\|_2^2+f(t)\sum_{q\geq -1}\lambda_q^{2s}\|u_q\|_2^2.
\end{equation}
\end{Lemma}
\pf
We first recall Bony's paraproduct decomposition 
\begin{equation}\notag
\begin{split}
\Delta_q(u\cdot\nabla u)=&\sum_{|q-p|\leq 2}\Delta_q(u_{\leq{p-2}}\cdot\nabla u_p)+
\sum_{|q-p|\leq 2}\Delta_q(u_{p}\cdot\nabla u_{\leq{p-2}})\\
&+\sum_{p\geq q-2}\Delta_q(u_p\cdot\nabla\tilde u_p).
\end{split}
\end{equation}
Thus, applying Bony's paraproduct above, the flux term $I$ can be decomposed as 
\begin{equation}\label{eq-i}
\begin{split}
I=
&\sum_{q\geq -1}\sum_{\substack{|q-p|\leq 2 \\ p\geq-1}}\lambda_q^{2s}\int_{\R^3}\Delta_q(u_{\leq p-2}\cdot\nabla u_{p})\cdot u_q\, dx\\
&+\sum_{q\geq -1}\sum_{\substack{|q-p|\leq 2 \\ p\geq-1}}\lambda_q^{2s}\int_{\R^3}\Delta_q(u_{p}\cdot\nabla u_{\leq{p-2}})\cdot u_q\, dx\\
&+\sum_{q\geq -1}\sum_{\substack{p\geq q-2 \\ p\geq-1}}\lambda_q^{2s}\int_{\R^3}\Delta_q(u_p\cdot\nabla\tilde u_p)\cdot u_q\, dx\\
=&I_{1}+I_{2}+I_{3},
\end{split}
\end{equation}
with 
\begin{equation}\label{eq-i1}
\begin{split}
I_{1}=&\sum_{q\geq -1}\sum_{\substack{|q-p|\leq 2 \\ p\geq-1}}\lambda_q^{2s}\int_{\R^3}[\Delta_q, u_{\leq{p-2}}\cdot\nabla] u_p\cdot u_q\, dx\\
&+\sum_{q\geq -1}\lambda_q^{2s}\int_{\R^3}(u_{\leq{q-2}}\cdot\nabla)  u_q \cdot u_q\, dx\\
&+\sum_{q\geq -1}\sum_{\substack{|q-p|\leq 2 \\ p\geq-1}}\lambda_q^{2s}\int_{\R^3}((u_{\leq{p-2}}-u_{\leq{q-2}})\cdot\nabla)\Delta_qu_p \cdot u_q\, dx\\
=&I_{11}+I_{12}+I_{13},
\end{split}
\end{equation}
where we used $\sum_{|q-p|\leq 2}\Delta_q u_p=u_q$.
One can see that $I_{12}$ vanishes since $\div\, u_{\leq q-2}=0$.
We now estimate $I_{11}$, where 
we used the commutator notation
\begin{equation}\notag
[\Delta_q, u_{\leq{p-2}}\cdot\nabla]u_p=\Delta_q(u_{\leq{p-2}}\cdot\nabla u_p)-(u_{\leq{p-2}}\cdot\nabla) \Delta_qu_p.
\end{equation}
Further splitting $I_{11}$ based on definition of $\Lambda_r$, we get
\begin{equation}\notag
 \begin{split}
|I_{11}|\leq &\sum_{q\geq -1}\sum_{\substack{|q-p|\leq 2 \\ p\geq-1}}\lambda_q^{2s}\int_{\R^3}\left|[\Delta_q, u_{\leq{p-2}}\cdot\nabla] u_p\cdot u_q\right|\, dx\\
\leq & \sum_{-1\leq p\leq Q+2}\sum_{\substack{|q-p|\leq 2 \\ q\geq-1}}\lambda_q^{2s}\int_{\R^3}\left|[\Delta_q, u_{\leq{p-2}}\cdot\nabla] u_p\cdot u_q\right|\, dx\\
&+\sum_{p> Q+2}\sum_{\substack{|q-p|\leq 2 \\ q\geq-1}}\lambda_q^{2s}\int_{\R^3}\left|[\Delta_q, u_{\leq Q}\cdot\nabla] u_p\cdot u_q\right|\, dx\\
&+\sum_{p> Q+2}\sum_{\substack{|q-p|\leq 2 \\ q\geq-1}}\lambda_q^{2s}\int_{\R^3}\left|[\Delta_q, u_{(Q,p-2]}\cdot\nabla] u_p\cdot u_q\right|\, dx\\
\equiv &I_{111}+I_{112}+I_{113}.
\end{split}
\end{equation}
Using the commutator estimate (\ref{est-comm}), H\"older's inequality, and  definition of $f(t)$, we obtain
\begin{equation}\notag
\begin{split}
I_{111}\lesssim &\sum_{-1\leq p\leq Q+2}\sum_{\substack{|q-p|\leq 2 \\ q\geq-1}}\lambda_q^{2s}\|\nabla u_{\leq p-2}\|_\infty\|u_p\|_2\|u_q\|_2\\
\lesssim & f(t) \sum_{-1\leq p\leq Q+2}\|u_p\|_2\sum_{\substack{|q-p|\leq 2 \\ q\geq-1}}\lambda_q^{2s}\|u_q\|_2\\
\lesssim & f(t) \sum_{q\geq -1}\lambda_q^{2s}\|u_q\|_2^2;
\end{split} 
\end{equation}
and similarly
\begin{equation}\notag
\begin{split}
I_{112}\lesssim &\sum_{ p> Q+2}\sum_{\substack{|q-p|\leq 2 \\ q\geq-1}}\lambda_q^{2s}\|\nabla u_{\leq Q}\|_\infty\|u_p\|_2\|u_q\|_2\\
\lesssim & f(t) \sum_{ p> Q+2}\|u_p\|_2\sum_{\substack{|q-p|\leq 2 \\ q\geq-1}}\lambda_q^{2s}\|u_q\|_2\\
\lesssim & f(t) \sum_{q> Q}\lambda_q^{2s}\|u_q\|_2^2.
\end{split}
\end{equation}
Now using H\"older's inequality, (\ref{est-comm}), definition of $\Lambda_r$, Bernstein's inequality, and Jensen's inequality, we deduce 
\begin{equation}\notag
\begin{split}
I_{113}
\lesssim &\sum_{ p> Q+2}\sum_{\substack{|q-p|\leq 2 \\ q\geq-1}}\lambda_q^{2s}\|[\Delta_q,u_{(Q,p-2]}\cdot\nabla]u_p\|_2\|u_q\|_2\\
\lesssim &\sum_{ p> Q+2}\sum_{\substack{|q-p|\leq 2 \\ q\geq-1}}\lambda_q^{2s}\|u_q\|_2\sum_{Q<p'\leq p-2}\lambda_{p'}\|u_{p'}\|_r\|u_p\|_{\frac{2r}{r-2}}\\
\lesssim & \sum_{ p> Q+2}\lambda_p^{\frac3r}\|u_p\|_2\sum_{\substack{|q-p|\leq 2 \\ q\geq-1}}\lambda_q^{2s}\|u_q\|_2\sum_{Q<p'\leq p-2}\lambda_{p'}\|u_{p'}\|_r\\
\lesssim & c_r\nu \sum_{ p> Q+2}\lambda_p^{\frac3r}\|u_p\|_2\sum_{\substack{|q-p|\leq 2 \\ q\geq-1}}\lambda_q^{2s}\|u_q\|_2\sum_{Q<p'\leq p-2}\lambda_{p'}^{2-\frac3r}\\
\lesssim & c_r\nu \sum_{ p> Q+2}\lambda_p^{2s+\frac3r}\|u_p\|_2^2\sum_{Q<p'\leq p-2}\lambda_{p'}^{2-\frac3r}\\
\lesssim & c_r\nu \sum_{ p> Q+2}\lambda_p^{2s+2}\|u_p\|_2^2\sum_{Q<p'\leq p-2}\lambda_{p'-p}^{2-\frac3r}\\
\lesssim & c_r\nu \sum_{ p> Q+2}\lambda_p^{2s+2}\|u_p\|_2^2,
\end{split}
\end{equation}
since $r\geq 2$. Combining the previous four inequalities lead to 

\begin{equation}\label{est-i11}
|I_{11}| \lesssim c_r\nu \sum_{ p> Q+2}\lambda_p^{2s+2}\|u_p\|_2^2+f(t) \sum_{q\geq -1}\lambda_q^{2s}\|u_q\|_2^2.
\end{equation}

To estimate $I_{13}$, we first split it as
\begin{equation}\notag
\begin{split}
|I_{13}|\leq & \sum_{q\geq -1}\sum_{\substack{|q-p|\leq 2 \\ p\geq-1}}\lambda_q^{2s}\left|\int_{\R^3}((u_{\leq{p-2}}-u_{\leq{q-2}})\cdot\nabla)\Delta_qu_p \cdot u_q\, dx\right|\\
= & \sum_{-1\leq q\leq Q}\sum_{\substack{|q-p|\leq 2 \\ p\geq-1}}\lambda_q^{2s}\left|\int_{\R^3}((u_{\leq{p-2}}-u_{\leq{q-2}})\cdot\nabla)\Delta_qu_p\cdot  u_q\, dx\right|\\
 & +\sum_{q>Q}\sum_{\substack{|q-p|\leq 2 \\ p\geq-1}}\lambda_q^{2s}\left|\int_{\R^3}((u_{\leq{p-2}}-u_{\leq{q-2}})\cdot\nabla)\Delta_qu_p \cdot u_q\, dx\right|\\
 \equiv & I_{131}+I_{132}.
 \end{split}
 \end{equation}
 Using H\"older's inequality, integration by parts and the fact $\nabla\cdot (u_{\leq{p-2}}-u_{\leq{q-2}})=0$, and definition of $f(t)$,  we obtain
 \begin{equation}\notag
 \begin{split}
 |I_{131}|\lesssim & \sum_{-1\leq q\leq Q}\lambda_q^{2s+1} \|u_q\|_{\infty}\sum_{\substack{|q-p|\leq 2 \\ p\geq-1}}\|u_{\leq{p-2}}-u_{\leq{q-2}}\|_2\|u_p\|_2\\
 \lesssim & f(t)\sum_{-1\leq q\leq Q}\lambda_q^{2s} \sum_{\substack{|q-p|\leq 2 \\ p\geq-1}}\|u_{\leq{p-2}}-u_{\leq{q-2}}\|_2\|u_p\|_2\\
\lesssim & f(t)\sum_{-1\leq q\leq Q}\lambda_q^{2s} \|u_q\|_2^2.
\end{split}
\end{equation}
Also, H\"older's inequality and  definition of $\Lambda_r$ yield
 \begin{equation}\notag
 \begin{split}
 |I_{132}|\lesssim  & \sum_{q> Q}\lambda_q^{2s+1} \|u_q\|_r\sum_{\substack{|q-p|\leq 2 \\ p\geq-1}}\|u_{\leq{p-2}}-u_{\leq{q-2}}\|_2\|u_p\|_{\frac{2r}{r-2}}\\
\lesssim 
 & c_r\nu \sum_{q> Q}\lambda_q^{2s+2-\frac3r}\sum_{\substack{|q-p|\leq 2 \\ p\geq-1}}\lambda_p^{\frac3r}\|u_{\leq{p-2}}-u_{\leq{q-2}}\|_2\|u_p\|_2\\
\lesssim & c_r\nu \sum_{q> Q}\lambda_q^{2s+2-\frac3r}\sum_{q-3\leq p\leq q+2}\lambda_p^{\frac3r}\|u_p\|_2^2\\
\lesssim & c_r\nu \sum_{q> Q-3}\lambda_q^{2s+2}\|u_q\|_2^2.  
\end{split}
\end{equation}
The last three inequalities together imply 
\begin{equation}\label{est-i13}
|I_{13}|\lesssim c_r\nu \sum_{q> Q-3}\lambda_q^{2s+2}\|u_q\|_2^2+f(t)\sum_{-1\leq q\leq Q}\lambda_q^{2s} \|u_q\|_2^2.
\end{equation}
Therefore, we conclude from (\ref{eq-i1}), (\ref{est-i11}),  (\ref{est-i13}) and the fact $I_{12}=0$ that
\begin{equation}\label{est-i1}
|I_{1}|\lesssim c_r\nu \sum_{q> Q-3}\lambda_q^{2s+2}\|u_q\|_2^2+f(t)\sum_{ q\geq-1}\lambda_q^{2s} \|u_q\|_2^2.
\end{equation}

Regarding $I_2$ in (\ref{eq-i}), we claim that it enjoys the same estimate as $I_{11}$ in (\ref{eq-i1}). Indeed, the commutator estimate (\ref{est-comm}) indicates that the commutator $[\Delta_q, u_{\leq{p-2}}\cdot\nabla]u_p$ appeared in $I_{11}$ has the effect of moving a derivative from $u_p$ to $u_{\leq{p-2}}$. Hence, $I_{11}$ has the same essential structure as $I_2$. Therefore, by (\ref{est-i11}), we have 
\begin{equation}\label{est-i2}
|I_2| \lesssim c_r\nu \sum_{ p> Q+2}\lambda_p^{2s+2}\|u_p\|_2^2+f(t) \sum_{q\geq -1}\lambda_q^{2s}\|u_q\|_2^2 ,\ \mbox {for } \ \ r\geq 2.
\end{equation}

In the end, we deal with $I_{3}$ by using H\"older's inequality, definition of $\Lambda_r$ and $f(t)$, Bernstein's inequality, and change order of the summations, to obtain
\begin{equation}\label{est-i3}
\begin{split}
|I_{3}|\lesssim &\sum_{q\geq -1}\sum_{p\geq q-2}\lambda_q^{2s}\int_{\mathbb R^3}|\Delta_q(u_p\otimes\tilde u_p)\nabla u_q|\, dx\\
\lesssim &\sum_{q> Q}\lambda_q^{2s+1}\|u_q\|_\infty\sum_{p\geq q-3}\|u_p\|_2^2
+\sum_{-1\leq q\leq Q}\lambda_q^{2s+1}\|u_q\|_\infty\sum_{p\geq q-3}\|u_p\|_2^2\\
\lesssim &\sum_{q> Q}\lambda_q^{2s+1+\frac3r}\|u_q\|_r\sum_{p\geq q-3}\|u_p\|_2^2
+f(t)\sum_{-1\leq q\leq Q}\lambda_q^{2s}\sum_{p\geq q-3}\|u_p\|_2^2\\
\lesssim &c_r\nu\sum_{p> Q}\lambda_p^{2s+2}\|u_p\|_2^2\sum_{Q< q\leq p+3}\lambda_{q-p}^{2s+2}
+f(t)\sum_{p\geq -1}\lambda_p^{2s}\|u_p\|_2^2\sum_{q\leq p+3}\lambda_{q-p}^{2s}\\
\lesssim &c_r\nu\sum_{q> Q}\lambda_q^{2s+2}\|u_q\|_2^2+f(t)\sum_{q\geq -1}\lambda_q^{2s}\|u_q\|_2^2.
\end{split}
\end{equation}
Finally, the estimate (\ref{est-i}) is a consequence of (\ref{eq-i}), (\ref{est-i1}), (\ref{est-i2}), and (\ref{est-i3}).

\cbdu

The other flux terms $J, K$ and $L$ will be estimated in a similar fashion. However, to fully exploit the cancellations among the flux, we estimate $J+L$ instead of handling $J$ and $L$ separately.  

\begin{Lemma}\label{le-J} For any $r$ and $s$ satisfying $\frac12<s<1$ and $2\leq r<\frac 3s$, we have
\begin{equation}\label{est-jl}
|J+L|\lesssim c_r\mu\sum_{q\geq -1}\lambda_q^{2s+2}\|b_q\|_2^2+f(t)\sum_{q\geq -1}\lambda_q^{2s}\|b_q\|_2^2.
\end{equation}
\end{Lemma}
\pf
Similarly, using Bony's paraproduct decomposition and the commutator notation, $J$ can be decomposed as
\begin{equation}\label{eq-j}
\begin{split}
J=
&-\sum_{q\geq -1}\sum_{\substack{|q-p|\leq 2 \\ p\geq-1}}\lambda_q^{2s}\int_{\R^3}\Delta_q(b_{\leq p-2}\cdot\nabla b_p)\cdot u_q\, dx\\
&-\sum_{q\geq -1}\sum_{\substack{|q-p|\leq 2 \\ p\geq-1}}\lambda_q^{2s}\int_{\R^3}\Delta_q(b_{p}\cdot\nabla b_{\leq{p-2}})\cdot u_q\, dx\\
&-\sum_{q\geq -1}\sum_{\substack{p\geq q-2 \\ p\geq-1}}\lambda_q^{2s}\int_{\R^3}\Delta_q(b_p\cdot\nabla\tilde b_p)\cdot u_q\, dx\\
=&J_{1}+J_{2}+J_{3},
\end{split}
\end{equation}
with 
\begin{equation}\label{eq-j1}
\begin{split}
J_{1}=&-\sum_{q\geq -1}\sum_{\substack{|q-p|\leq 2 \\ p\geq-1}}\lambda_q^{2s}\int_{\R^3}[\Delta_q, b_{\leq{p-2}}\cdot\nabla] b_p\cdot u_q\, dx\\
&-\sum_{q\geq -1}\sum_{\substack{|q-p|\leq 2 \\ p\geq-1}}\lambda_q^{2s}\int_{\R^3}(b_{\leq{q-2}}\cdot\nabla) \Delta_q b_p \cdot u_q\, dx\\
&-\sum_{q\geq -1}\sum_{\substack{|q-p|\leq 2 \\ p\geq-1}}\lambda_q^{2s}\int_{\R^3}((b_{\leq{p-2}}-b_{\leq{q-2}})\cdot\nabla)\Delta_qb_p \cdot u_q\, dx\\
=&J_{11}+J_{12}+J_{13};
\end{split}
\end{equation}
while the flux $L$ can be decomposed as
\begin{equation}\label{eq-l}
\begin{split}
L=
&-\sum_{q\geq -1}\sum_{\substack{|q-p|\leq 2 \\ p\geq-1}}\lambda_q^{2s}\int_{\R^3}\Delta_q(b_{\leq p-2}\cdot\nabla u_p)\cdot b_q\, dx\\
&-\sum_{q\geq -1}\sum_{\substack{|q-p|\leq 2 \\ p\geq-1}}\lambda_q^{2s}\int_{\R^3}\Delta_q(b_{p}\cdot\nabla u_{\leq{p-2}})\cdot b_q\, dx\\
&-\sum_{q\geq -1}\sum_{\substack{p\geq q-2 \\ p\geq-1}}\lambda_q^{2s}\int_{\R^3}\Delta_q(\tilde b_p\cdot\nabla u_p)\cdot b_q\, dx\\
=&L_{1}+L_{2}+L_{3},
\end{split}
\end{equation}
with 
\begin{equation}\label{eq-l1}
\begin{split}
L_{1}=&-\sum_{q\geq -1}\sum_{\substack{|q-p|\leq 2 \\ p\geq-1}}\lambda_q^{2s}\int_{\R^3}[\Delta_q,b_{\leq{p-2}}\cdot\nabla] u_p\cdot b_q\, dx\\
&-\sum_{q\geq -1}\sum_{\substack{|q-p|\leq 2 \\ p\geq-1}}\lambda_q^{2s}\int_{\R^3}(b_{\leq{q-2}}\cdot\nabla)\Delta_q u_p \cdot b_q\, dx\\
&-\sum_{q\geq -1}\sum_{\substack{|q-p|\leq 2 \\ p\geq-1}}\lambda_q^{2s}\int_{\R^3}\left((b_{\leq{p-2}}-b_{\leq{q-2}})\cdot\nabla\right)\Delta_qu_p\cdot  b_q\, dx\\
=&L_{11}+L_{12}+L_{13}.
\end{split}
\end{equation}
We shall show that the term $J_{12}$ cancels $L_{12}$. Indeed, using the fact $\sum_{|q-p|\leq 2}\Delta_pb_q=b_q$ and integration by parts, we notice 
\[\sum_{\substack{|q-p|\leq 2 \\ p\geq-1}}\lambda_q^{2s}\int_{\R^3}(b_{\leq{q-2}}\cdot\nabla)\Delta_q b_p \cdot b_q\, dx
=\lambda_q^{2s}\int_{\R^3}(b_{\leq{q-2}}\cdot\nabla) b_q \cdot b_q\, dx=0;\]
and analogously, 
\[\sum_{\substack{|q-p|\leq 2 \\ p\geq-1}}\lambda_q^{2s}\int_{\R^3}(b_{\leq{q-2}}\cdot\nabla)\Delta_q u_p \cdot u_q\, dx
=\lambda_q^{2s}\int_{\R^3}(b_{\leq{q-2}}\cdot\nabla) u_q \cdot u_q\, dx=0.\]
Therefore, we deduce
\begin{equation}\label{eq-j12l12}
\begin{split}
J_{12}+L_{12}=&-\sum_{q\geq -1}\sum_{\substack{|q-p|\leq 2 \\ p\geq-1}}\lambda_q^{2s}\int_{\R^3}(b_{\leq{q-2}}\cdot\nabla)\Delta_q b_p \cdot (u_q+b_q)\, dx\\
&-\sum_{q\geq -1}\sum_{\substack{|q-p|\leq 2 \\ p\geq-1}}\lambda_q^{2s}\int_{\R^3}(b_{\leq{q-2}}\cdot\nabla)\Delta_q u_p \cdot (b_q+u_q)\, dx\\
=&-\sum_{q\geq -1}\sum_{\substack{|q-p|\leq 2 \\ p\geq-1}}\lambda_q^{2s}\int_{\R^3}(b_{\leq{q-2}}\cdot\nabla)(u_p+b_p) \cdot (u_q+b_q)\, dx\\
=&0.
\end{split}
\end{equation}

The other terms appeared in (\ref{eq-j})-(\ref{eq-l1}) are estimated as follows.
First, we further split $J_{11}$ as 
\begin{equation}\notag
\begin{split}
|J_{11}|\leq & \sum_{q>Q}\sum_{\substack{|q-p|\leq 2 \\ p\geq-1}}\lambda_q^{2s}\int_{\R^3}\left|[\Delta_q, b_{\leq{p-2}}\cdot\nabla] b_p\cdot u_q\right|\, dx\\
&+\sum_{-1\leq q\leq Q}\sum_{\substack{|q-p|\leq 2 \\ p\geq-1}}\lambda_q^{2s}\int_{\R^3}\left|[\Delta_q, b_{\leq{p-2}}\cdot\nabla] b_p\cdot u_q\right|\, dx\\
\equiv & J_{111}+J_{112}.
\end{split}
\end{equation}
Then, by H\"older's inequality, (\ref{est-comm}), definition of $\Lambda_r$, and Jensen's inequality, it follows that  
\begin{equation}\notag
\begin{split}
|J_{111}|
\lesssim&\sum_{q> Q}\lambda_q^{2s}\|u_q\|_r\sum_{\substack{|q-p|\leq 2 \\ p\geq-1}}\|b_p\|_2\sum_{p'\leq p-2}\lambda_{p'}\|b_{p'}\|_{\frac{2r}{r-2}}\\
\lesssim&c_r\mu \sum_{q> Q}\lambda_q^{2s+1-\frac3r}\sum_{\substack{|q-p|\leq 2 \\ p\geq-1}}\|b_p\|_2\sum_{p'\leq q}\lambda_{p'}^{1+\frac3r}\|b_{p'}\|_2\\
\lesssim&c_r\mu \sum_{q> Q-2}\lambda_q^{2s+1-\frac3r}\|b_q\|_2\sum_{p'\leq q}\lambda_{p'}^{1+\frac3r}\|b_{p'}\|_2\\
\lesssim&c_r\mu\sum_{q> Q-2}\lambda_q^{s+1}\|b_q\|_2\sum_{p'\leq q}\lambda_{p'}^{s+1}\|b_{p'}\|_2\lambda_{q-p'}^{s-\frac 3r}\\
\lesssim&c_r\mu\sum_{q\geq -1}\lambda_q^{2s+2}\|b_q\|_2^2,
\end{split}
\end{equation}
where we needed  $2\leq r<\frac 3s$.
Also, by H\"older's inequality, definition of $f(t)$ and Jensen's inequality, 
\begin{equation}\notag
\begin{split}
|J_{112}|
\lesssim &\sum_{-1\leq q\leq Q}\lambda_q^{2s}\|u_q\|_\infty\sum_{\substack{|q-p|\leq 2 \\ p\geq-1}}\|b_p\|_2\sum_{p'\leq p-2}\lambda_{p'}\|b_{p'}\|_2\\
\lesssim &f(t)\sum_{-1\leq q\leq Q}\lambda_q^{2s-1}\sum_{\substack{|q-p|\leq 2 \\ p\geq-1}}\|b_p\|_2\sum_{p'\leq q}\lambda_{p'}\|b_{p'}\|_2\\
\lesssim&f(t)\sum_{-1\leq q\leq Q+2}\lambda_q^{2s-1}\|b_q\|_2\sum_{p'\leq q}\lambda_{p'}\|b_{p'}\|_2\\
\lesssim&f(t)\sum_{-1\leq q\leq Q+2}\lambda_q^{s}\|b_q\|_2\sum_{p'\leq q}\lambda_{p'}^s\|b_{p'}\|_2\lambda_{p'-q}^{1-s}\\
\lesssim&f(t)\sum_{-1\leq q\leq Q+2}\lambda_q^{2s}\|b_q\|_2^2,
\end{split}
\end{equation}
where we used $\frac 12< s<1$. The last three inequalities together imply that for $\frac 12< s<1$ and $2\leq r<\frac 3s$,
\begin{equation}\label{est-j11}
|J_{11}|\lesssim c_r\mu\sum_{q\geq -1}\lambda_q^{2s+2}\|b_q\|_2^2+ f(t) \sum_{-1\leq q\leq Q+2}\lambda_q^{2s}\|b_q\|_2^2.
\end{equation}
Similar analysis of employing wavenumber splitting, H\"older's inequality, Bernstein's inequality, definition of $\Lambda_r$ and $f(t)$ yields, for $r\geq 2$ and $s\geq 0$
\begin{equation}\label{est-j13}
\begin{split}
|J_{13}|\leq & \sum_{q\geq -1}\sum_{\substack{|q-p|\leq 2 \\ p\geq-1}}\lambda_q^{2s}\int_{\R^3}|(b_{\leq{p-2}}-b_{\leq{q-2}})\cdot\nabla\Delta_qb_p u_q|\, dx\\
\leq & \sum_{q>Q}\sum_{\substack{|q-p|\leq 2 \\ p\geq-1}}\lambda_q^{2s+1}\|u_q\|_r\|b_{\leq{p-2}}-b_{\leq{q-2}}\|_2\|b_p\|_{\frac{2r}{r-2}}\\
& +\sum_{-1\leq q\leq Q}\sum_{\substack{|q-p|\leq 2 \\ p\geq-1}}\lambda_q^{2s+1}\|u_q\|_\infty\|b_{\leq{p-2}}-b_{\leq{q-2}}\|_2\|b_p\|_2\\
\lesssim & c_r\mu \sum_{q>Q}\lambda_q^{2s+2-\frac3r}\sum_{\substack{|q-p|\leq 2 \\ p\geq-1}}\lambda_p^{\frac3r}\|b_{\leq{p-2}}-b_{\leq{q-2}}\|_2\|b_p\|_2\\
& +f(t)\sum_{-1\leq q\leq Q}\lambda_q^{2s}\sum_{\substack{|q-p|\leq 2 \\ p\geq-1}}\|b_{\leq{p-2}}-b_{\leq{q-2}}\|_2\|b_p\|_2\\
\lesssim & c_r\mu \sum_{q>Q}\lambda_q^{2s+2-\frac3r}\sum_{q-3\leq p\leq q+2}\lambda_p^{\frac3r}\|b_p\|_2^2 +f(t)\sum_{-1\leq q\leq Q+2}\lambda_q^{2s}\|b_q\|_2^2\\
\lesssim & c_r\mu \sum_{q>Q-3}\lambda_q^{2s+2}\|b_p\|_2^2 +f(t)\sum_{-1\leq q\leq Q+2}\lambda_q^{2s}\|b_q\|_2^2.
\end{split}
\end{equation}
Combining (\ref{eq-j1}), (\ref{est-j11}) and (\ref{est-j13}), we obtain
\begin{equation}\label{est-j1}
|J_1-J_{12}|\lesssim c_r\mu\sum_{q\geq -1}\lambda_q^{2s+2}\|b_q\|_2^2+ f(t) \sum_{-1\leq q\leq Q+2}\lambda_q^{2s}\|b_q\|_2^2.
\end{equation} 
Notice that $J_{2}$ enjoys the same estimate as $J_{11}$,  due to the similar reason that $I_{2}$ enjoys the same estimate as $I_{11}$. Thus, thanks to the estimate (\ref{est-j11}), we have for $\frac 12< s<1$ and $2\leq r<\frac 3s$,
\begin{equation}\label{est-j2}
|J_2|\lesssim c_r\mu\sum_{q\geq -1}\lambda_q^{2s+2}\|b_q\|_2^2+ f(t) \sum_{-1\leq q\leq Q+2}\lambda_q^{2s}\|b_q\|_2^2.
\end{equation} 
Applying H\"older's inequality, definition of $\Lambda_r$ and $f(t)$, Bernstein's inequality, and change order of the summations, the last term $J_{3}$ of flux $J$ has the estimate
\begin{equation}\label{est-j3}
\begin{split}
|J_{3}|
\lesssim &\sum_{q\geq -1}\sum_{p\geq q-2}\lambda_q^{2s}\int_{\mathbb R^3}|\Delta_q(b_p\otimes\tilde b_p)\nabla u_q|\, dx\\
\lesssim &\sum_{q> Q}\lambda_q^{2s+1}\|u_q\|_\infty\sum_{p\geq q-3}\|b_p\|_2^2
+\sum_{-1\leq q\leq Q}\lambda_q^{2s+1}\|u_q\|_\infty\sum_{p\geq q-3}\|b_p\|_2^2\\
\lesssim &\sum_{q> Q}\lambda_q^{2s+1+\frac3r}\|u_q\|_r\sum_{p\geq q-3}\|b_p\|_2^2
+f(t)\sum_{-1\leq q\leq Q}\lambda_q^{2s}\sum_{p\geq q-3}\|b_p\|_2^2\\
\lesssim &c_r\mu\sum_{p> Q}\lambda_p^{2s+2}\|b_p\|_2^2\sum_{Q< q\leq p+3}\lambda_{q-p}^{2s+2}\\
&+f(t)\sum_{p\geq -1}\lambda_p^{2s}\|b_p\|_2^2\sum_{-1\leq q\leq p+3}\lambda_{q-p}^{2s}\\
\lesssim &c_r\mu\sum_{q> Q}\lambda_q^{2s+2}\|b_q\|_2^2+f(t)\sum_{q\geq -1}\lambda_q^{2s}\|b_q\|_2^2.
\end{split}
\end{equation}
Therefore, the equation (\ref{eq-j}) together with estimates (\ref{est-j1}), (\ref{est-j2}) and (\ref{est-j3}) leads to
\begin{equation}\label{est-j}
|J-J_{12}|\lesssim c_r\mu\sum_{q\geq -1}\lambda_q^{2s+2}\|b_q\|_2^2+f(t)\sum_{q\geq -1}\lambda_q^{2s}\|b_q\|_2^2.
\end{equation}

Now we turn to estimates of terms in $L$ appeared in (\ref{eq-l}) and (\ref{eq-l1}).
Notice that $L_{11}$ can be estimated as $J_{11}$. 
Thus, it follows from (\ref{est-j11}) that
\begin{equation}\label{est-l11}
|L_{11}|\lesssim c_r\mu\sum_{q\geq -1}\lambda_q^{2s+2}\|b_q\|_2^2+f(t)\sum_{-1\leq q\leq Q+2}\lambda_q^{2s}\|b_q\|_2^2.
\end{equation}
After using integration by parts, the term $L_{13}$ can be estimated similarly to $J_{13}$. Hence, we have from (\ref{est-j13})
\begin{equation}\label{est-l13}
|L_{13}|\lesssim c_r\mu\sum_{q> Q-3}\lambda_q^{2s+2}\|b_q\|_2^2+f(t)\sum_{-1\leq q\leq Q+2}\lambda_q^{2s}\|b_q\|_2^2.
\end{equation} 
To estimate $L_{2}$, we split the summation first as
\begin{equation}\notag
\begin{split}
|L_{2}|\leq&\sum_{-1\leq p\leq Q+2}\sum_{\substack{|q-p|\leq 2 \\ p\geq-1}}\lambda_q^{2s}\int_{\mathbb R^3}\left|\Delta_q(b_p\cdot\nabla u_{\leq p-2})\cdot b_q\right|\,dx\\
&+\sum_{p>Q+2}\sum_{\substack{|q-p|\leq 2 \\ p\geq-1}}\lambda_q^{2s}\int_{\mathbb R^3}\left|\Delta_q(b_p\cdot\nabla u_{\leq Q})\cdot b_q\right|\,dx\\
&+\sum_{p>Q+2}\sum_{\substack{|q-p|\leq 2 \\ p\geq-1}}\lambda_q^{2s}\int_{\mathbb R^3}\left|\Delta_q(b_p\cdot\nabla u_{(Q,p-2]})\cdot b_q\right|\,dx\\
\equiv &\ L_{21}+L_{22}+L_{23}.
\end{split}
\end{equation}
Then using H\"older's inequality and  definition of $f(t)$ we arrive at
\begin{equation}\notag
\begin{split}
L_{21}\lesssim & \sum_{-1\leq p\leq Q+2}\|\nabla u_{\leq p-2}\|_\infty\|b_p\|_2\sum_{\substack{|q-p|\leq 2 \\ p\geq-1}}\lambda_q^{2s}\|b_q\|_2\\
\lesssim & f(t)\sum_{-1\leq p\leq Q+2}\|b_p\|_2\sum_{\substack{|q-p|\leq 2 \\ p\geq-1}}\lambda_q^{2s}\|b_q\|_2\\
\lesssim & f(t)\sum_{q\geq -1}\lambda_q^{2s}\|b_q\|_2^2,
\end{split}
\end{equation}
and
\begin{equation}\notag
\begin{split}
L_{22} \lesssim &\sum_{p> Q+2}\|\nabla u_{\leq Q}\|_\infty\|b_p\|_2\sum_{\substack{|q-p|\leq 2 \\ p\geq-1}}\lambda_q^{2s}\|b_q\|_2\\
\lesssim & f(t)\sum_{p> Q+2}\|b_p\|_2\sum_{\substack{|q-p|\leq 2 \\ p\geq-1}}\lambda_q^{2s}\|b_q\|_2\\
\lesssim & f(t)\sum_{q> Q}\lambda_q^{2s}\|b_q\|_2^2.
\end{split}
\end{equation}
Using H\"older's inequality, definition of $\Lambda_r$ and Jensen's inequality, we deduce
\begin{equation}\notag
\begin{split}
L_{23} \lesssim &\sum_{p> Q+2}\|\nabla u_{(Q,p-2]}\|_r\|b_p\|_{\frac{2r}{r-2}}\sum_{\substack{|q-p|\leq 2 \\ p\geq-1}}\lambda_q^{2s}\|b_q\|_2\\
\lesssim &\sum_{p> Q+2}\lambda_p^{\frac3r}\|b_p\|_2\sum_{\substack{|q-p|\leq 2 \\ p\geq-1}}\lambda_q^{2s}\|b_q\|_2\sum_{Q<p'\leq p-2}\lambda_{p'}\|u_{p'}\|_r\\
\lesssim & c_r\mu \sum_{p> Q}\lambda_p^{2s+\frac3r}\|b_p\|_2^2\sum_{Q<p'\leq p-2}\lambda_{p'}^{2-\frac3r}\\
\lesssim & c_r\mu \sum_{p> Q}\lambda_p^{2s+2}\|b_p\|_2^2\sum_{Q<p'\leq p-2}\lambda_{p'-p}^{2-\frac3r}\\
\lesssim & c_r\mu \sum_{p> Q}\lambda_p^{2s+2}\|b_p\|_2^2,
\end{split}
\end{equation}
since $r\geq 2$.  The last four inequalities together indicate 
\begin{equation}\label{est-l2}
|L_2|\lesssim c_r\mu \sum_{p> Q}\lambda_p^{2s+2}\|b_p\|_2^2+f(t)\sum_{q\geq -1}\lambda_q^{2s}\|b_q\|_2^2.
\end{equation}

To estimate $L_{3}$, we use integration by parts first, and then split it as
\begin{equation}\label{ineq-l3}
\begin{split}
|L_{3}|\leq&\sum_{p\geq -1}\lambda_q^{2s}\sum_{-1\leq q\leq p+2}\left|\int_{\mathbb R^3}\Delta_q(\tilde b_p\otimes u_p)\nabla b_q \, dx\right|\\
=&\sum_{p>Q}\lambda_q^{2s}\sum_{-1\leq q\leq p+2}\left|\int_{\mathbb R^3}\Delta_q(\tilde b_p\otimes u_p)\nabla b_q \, dx\right|\\
&+\sum_{-1\leq p\leq Q}\lambda_q^{2s}\sum_{-1\leq q\leq p+2}\left|\int_{\mathbb R^3}\Delta_q(\tilde b_p\otimes u_p)\nabla b_q \, dx\right|\\
\equiv & L_{31}+L_{32}.
\end{split}
\end{equation}
Proceeding as in the case of $I_3$, we obtain
\begin{equation}\notag
\begin{split}
|L_{31}|
\lesssim &\sum_{p> Q}\|u_p\|_\infty\|b_p\|_2\sum_{-1\leq q\leq p+2}\lambda_q^{2s+1}\|b_q\|_2\\
\lesssim &c_r\mu\sum_{p> Q}\lambda_p\|b_p\|_2\sum_{-1\leq q\leq p+2}\lambda_q^{2s+1}\|b_q\|_2\\
\lesssim &c_r\mu\sum_{p> Q}\lambda_p^{s+1}\|b_p\|_2\sum_{-1\leq q\leq p+2}\lambda_q^{s+1}\|b_q\|_2\lambda_{q-p}^{s}\\
\lesssim &c_r\mu\sum_{q\geq -1}\lambda_q^{2s+2}\|b_q\|_2^2;
\end{split}
\end{equation}
and 
\begin{equation}\notag
\begin{split}
|L_{32}|
\lesssim &\sum_{-1\leq p\leq Q}\|u_p\|_\infty\|b_p\|_2\sum_{-1\leq q\leq p+2}\lambda_q^{2s+1}\|b_q\|_2\\
\lesssim &f(t)\sum_{-1\leq p\leq Q}\lambda_p^{-1}\|b_p\|_2\sum_{-1\leq q\leq p+2}\lambda_q^{2s+1}\|b_q\|_2\\
\lesssim &f(t)\sum_{-1\leq p\leq Q}\lambda_p^{s}\|b_p\|_2\sum_{-1\leq q\leq p+2}\lambda_q^{s}\|b_q\|_2\lambda_{q-p}^{s+1}\\
\lesssim &f(t)\sum_{-1\leq q\leq Q+2}\lambda_q^{2s}\|b_q\|_2^2.
\end{split}
\end{equation}
Putting together the last three inequalities, we obtain
\begin{equation}\label{est-l3}
|L_3|\lesssim c_r\mu\sum_{q\geq -1}\lambda_q^{2s+2}\|b_q\|_2^2+f(t)\sum_{-1\leq q\leq Q+2}\lambda_q^{2s}\|b_q\|_2^2.
\end{equation}
The combination of equations (\ref{eq-l})-(\ref{eq-l1}) and inequalities (\ref{est-l11})-(\ref{est-l3}) yields
\begin{equation}\label{est-l}
|L-L_{12}|\lesssim c_r\mu\sum_{q\geq -1}\lambda_q^{2s+2}\|b_q\|_2^2+f(t)\sum_{ q\geq -1}\lambda_q^{2s}\|b_q\|_2^2.
\end{equation}

Thus, the estimate (\ref{est-jl}) is a consequence of (\ref{eq-j12l12}), (\ref{est-j}) and (\ref{est-l}). 

\cbdu

\begin{Lemma}\label{le-K} For $\frac 12< s<1$ and $2\leq r<\frac 3s$, we have
\begin{equation}\label{est-k}
|K|\lesssim c_r\mu\sum_{q\geq -1}\lambda_q^{2s+2}\|b_q\|_2^2+f(t)\sum_{q\geq -1}\lambda_q^{2s}\|b_q\|_2^2.
\end{equation}
\end{Lemma}
\pf
As before, we start with Bony's decomposition,
\begin{equation}\label{eq-k}
\begin{split}
K=
&\sum_{q\geq -1}\sum_{\substack{|q-p|\leq 2 \\ p\geq-1}}\lambda_q^{2s}\int_{\R^3}\Delta_q(u_{\leq p-2}\cdot\nabla b_p)\cdot b_q\, dx\\
&+\sum_{q\geq -1}\sum_{\substack{|q-p|\leq 2 \\ p\geq-1}}\lambda_q^{2s}\int_{\R^3}\Delta_q(u_{p}\cdot\nabla b_{\leq{p-2}})\cdot b_q\, dx\\
&+\sum_{q\geq -1}\sum_{\substack{p\geq q-2 \\ p\geq-1}}\lambda_q^{2s}\int_{\R^3}\Delta_q(u_p\cdot\nabla\tilde b_p)\cdot b_q\, dx\\
=&K_{1}+K_{2}+K_{3},
\end{split}
\end{equation}
with 
\begin{equation}\label{eq-k1}
\begin{split}
K_{1}=&\sum_{q\geq -1}\sum_{\substack{|q-p|\leq 2 \\ p\geq-1}}\lambda_q^{2s}\int_{\R^3}[\Delta_q,u_{\leq{p-2}}\cdot\nabla] b_p\cdot b_q\, dx\\
&+\sum_{q\geq -1}\lambda_q^{2s}\int_{\R^3}(u_{\leq{q-2}}\cdot\nabla) b_q \cdot b_q\, dx\\
&+\sum_{q\geq -1}\sum_{\substack{|q-p|\leq 2 \\ p\geq-1}}\lambda_q^{2s}\int_{\R^3}((u_{\leq{p-2}}-u_{\leq{q-2}})\cdot\nabla)\Delta_qb_p\cdot  b_q\, dx\\
=&K_{11}+K_{12}+K_{13}.
\end{split}
\end{equation}
Here we used $\sum_{|q-p|\leq 2}\Delta_pb_q=b_q$ to obtain $K_{12}$.
We further notice that $K_{12}=0$ by applying integration by parts and the fact $\nabla\cdot u_{\leq{q-2}}=0$. 

Comparing $K_{11}$ in (\ref{eq-k1}) and $I_{11}$ in (\ref{eq-i1}), one can see that $K_{11}$ shares the same structure as $I_{11}$ (with $u_p$ and $u_q$ replaced by $b_p$ and $b_q$).
Thus, in view of (\ref{est-i11}), we have for $r\geq 2$
\begin{equation}\notag
|K_{11}|\lesssim c_r\mu\sum_{q>Q+2}\lambda_q^{2s+2}\|b_q\|_2^2+f(t)\sum_{q\geq -1}\lambda_q^{2s}\|b_q\|_2^2.
\end{equation}
Similarly, examining the equations (\ref{eq-k1}) and (\ref{eq-j1}), (\ref{eq-k}) and (\ref{eq-j}) respectively, we realize that $K_{13}$ and $J_{13}$ have essentially the same structure, and so do $K_{2}$ and $J_{2}$. Hence, it follows from (\ref{est-j13}) and (\ref{est-j2}) that for $\frac 12< s<1$ and $2\leq r<\frac 3s$
\begin{equation}\notag
|K_{13}|+|K_{2}|\lesssim c_r\mu\sum_{q\geq -1}\lambda_q^{2s+2}\|b_q\|_2^2+f(t)\sum_{-1\leq q\leq Q+2}\lambda_q^{2s}\|b_q\|_2^2.
\end{equation} 
Regarding $K_{3}$, we first use integration by parts 
\begin{equation}\notag
\begin{split}
|K_{3}|\leq&\sum_{p\geq -1}\lambda_q^{2s}\sum_{-1\leq q\leq p+2}\left|\int_{\mathbb R^3}\Delta_q(u_p\otimes\tilde b_p)\nabla b_q \, dx\right|\\
\end{split}
\end{equation}
and then notice that it has the same structure as $L_3$ in (\ref{ineq-l3}). Thus, it follows from (\ref{est-l3}) that for $r\geq 2$ and $s\geq 0$
\begin{equation}\notag
|K_3|\lesssim c_r\mu\sum_{q\geq -1}\lambda_q^{2s+2}\|b_q\|_2^2+f(t)\sum_{-1\leq q\leq Q+2}\lambda_q^{2s}\|b_q\|_2^2.
\end{equation}
Combining the estimates above and equations (\ref{eq-k}) and (\ref{eq-k1}) leads to the conclusion of the lemma.

\cbdu

In order to finish the proof of Proposition \ref{prop}, we conclude from inequality (\ref{ineq-ubq}) and estimates in Lemmas \ref{le-I}-\ref{le-K} that for some small enough constant $\tilde c$
we have
\begin{equation}\notag
\frac{d}{dt}\sum_{q\geq -1}\lambda_q^{2s}\left(\|u_q\|_2^2+\|b_q\|_2^2\right)
\lesssim f(t)\sum_{q\geq -1}\lambda_q^{2s}\left(\|u_q\|_2^2+\|b_q\|_2^2\right),
\end{equation}
provided $c_r \leq \tilde c$. This means that there exists an adimensional constant $C=C(s)$, such that
\begin{equation}\notag
\frac{d}{dt}\left(\|u\|_{ H^s}^2+\|b\|_{H^s}^2\right)
\leq Cf(t)\left(\|u\|_{ H^s}^2+\|b\|_{ H^s}^2\right), 
\end{equation}
which is the estimate (\ref{energy-m}).
Notice that the requirements $\frac 12< s<1$ and $2\leq r<\frac 3s$ in Lemmas \ref{le-J} and \ref{le-K} 
indicate that (\ref{energy-m}) holds for any $2\leq r<6$.
However, when $b \equiv 0$, we have $J \equiv K \equiv L =0$, and the range for $r$ extends to $[2, \infty]$.\\

\noindent {\bf Step (ii):}
First note that by Gr\"onwall's inequality  $\|u\|_{ H^s}^2+\|b\|_{ H^s}^2$ is 
bounded on $[T/2,T)$ provided $f\in L^1(0,T)$.  Our next goal is to weaken the
condition $f \in L^1$ to (\ref{criterion}). For this purpose, we will prove 

\begin{Proposition}\label{prop2}
Let $C$ be the constant in (\ref{energy-m}) and $\bar c=\frac{1}{2C}(s-\frac12)\ln 2$.
Assume that
\begin{equation}\label{criterion2}
\limsup_{q \to \infty} \int_{T/2}^T 1_{q\leq Q(\tau)}\lambda_q\|u_q\|_\infty \, d\tau \leq \bar c.
\end{equation}
Let $q^*$ be an integer such that
\[
\int_{T/2}^T 1_{q\leq Q(\tau)}\lambda_q\|u_q\|_\infty \, d\tau \leq 2\bar c,
\qquad \forall q > q^*.
\]
There exists a constant 
\[M_{q^*}\sim q^*\lambda_{q^*}^{\frac52}(\|u(0)\|_2+\|b(0)\|_2),
\]
such that 
\begin{equation}\label{f1-w}
\exp\left(\int_{T/2}^tCf(\tau)d\tau\right) \leq e^{TCM_{q^*}}  \min\{\nu,\mu\}^{-1} \sup_{\tau \in [T/2,t]}\|u(\tau)\|_{H^s}.
\end{equation}

\end{Proposition}
\pf
We first split $f$ as
\[
f(t)\leq f_{\leq q^*} + f_{>q^*}(t),
\]
with
\[
 f_{\leq q^*} =   \sum_{q \leq q^*} \lambda_q \|u_q\|_\infty,
\qquad f_{> q^*}(t)= 1_{q*<Q(t)}\sum_{q^*<q\leq Q(t)} \lambda_q\|u_q\|_\infty.
\]
Thanks to Bernstein's and energy inequalities, we have
\[
f_{\leq q^*}(t) \lesssim q^*\lambda_{q^*}^{\frac52} (\|u(0)\|_2 + \|b(0)\|_2) \leq M_{q^*}, \qquad t \in (0,T).
\]
It follows that
\begin{equation}\label{f1-low}
\exp\left(\int_{T/2}^t Cf_{\leq q^*}(\tau) \, d\tau\right) \leq e^{T C M_{q^*}}.
\end{equation}

By definition of $\Lambda_r(t)$ (see (\ref{Q})) and Bernstein's inequality we have
\[
c_r\min\{\nu,\mu\} \Lambda_r^{1-\frac{3}{r}} \leq \|u_Q\|_r
\lesssim \Lambda_r^{\frac{3}{2}-\frac{3}{r}} \|u_Q\|_2,
\]
provided $\Lambda_r>1$. Therefore
\[
c_r\min\{\nu,\mu\} \Lambda_r^{\varepsilon}(t) \lesssim \Lambda_r^{\frac{1}{2}+\varepsilon}\|u_Q\|_2 \lesssim 
 \|u\|_{H^{s}},
\]
with $\varepsilon = s-\frac{1}{2}>0$, which leads to
\begin{equation} \label{Qbar}
2^{\varepsilon \bar Q(t)}  \lesssim  \frac{1}{c_r\min\{\nu,\mu\}} \sup_{\tau \in [T/2,t]}\|u(\tau)\|_{H^s},
\end{equation}
where
\[
\bar Q(t) = \sup_{\tau \in (T/2,t]} Q(\tau).
\]

Then we can estimate the integral of $f_{> q*}$ over $(T/2, T)$ as follows, 
\[
\begin{split}
\int_{T/2}^t f_{>q^*}(\tau) \, d\tau &=  \int_{T/2}^t 1_{q*<Q(\tau)}\sum_{q^*<q\leq Q(\tau)} \lambda_q\|u_q\|_\infty \, d\tau\\
& =  \int_{T/2}^t 1_{q*<Q(\tau)} 1_{q\leq Q(\tau)}\sum_{q^*<q\leq Q(\tau)} \lambda_q\|u_q\|_\infty \, d\tau\\
&\leq \sum_{q^*<q\leq \bar Q(t)} \int_{T/2}^t 1_{q\leq Q(\tau)} \lambda_q\|u_q\|_\infty \, d\tau\\
&\leq \bar Q(t) \sup_{q>q^*}\int_{T/2}^T 1_{q\leq Q(\tau)}\lambda_q\|u_q\|_\infty \, d\tau\\
&\leq \bar Q(t) 2\bar c\\
&= \frac{\bar Q(t) \varepsilon \ln 2}{C}.
\end{split}
\]
Combining the last inequality with (\ref{Qbar}) yields
\begin{equation}\label{f1-high}
\begin{split}
\exp\left(\int_{T/2}^t Cf_{>q^*}(\tau) \, d\tau\right) &\leq 2^{\varepsilon \bar Q(t)}
\lesssim c_r^{-1}\min\{\nu,\mu\}^{-1} \sup_{\tau \in [T/2,t]}\|u(\tau)\|_{H^s}.
\end{split}
\end{equation}
Therefore, the estimate (\ref{f1-w}) is a consequence of (\ref{f1-low}) and (\ref{f1-high}).

\cbdu

\noindent{\bf{Step (iii):}} With the preparations done in Step (i) and Step (ii), we are ready to finish the proof of Theorem \ref{thm}.

First, notice that Lebesgue norms of Littlewood-Paley projections of $\nabla\times u$ and $\nabla u$ are equivalent. Indeed, since $\nabla \cdot u_q =0$, there exists a vector potential $A_q$ defined as $A_q=(-\Delta)^{-1} (\nabla\times u_q)$; 
consequently, $\nabla \times u_q = - \Delta A_q$ and $\nabla\cdot A_q=0$. Thus, we have
\begin{equation}\label{equiv-u}
\begin{split}
\|\nabla u_q\|_p\sim\lambda_q \|u_q\|_p = &\lambda_q \|\nabla \times A_q\|_p \\
\lesssim &\lambda_q^2 \|A_q\|_p \sim \| \Delta A_q\|_p = \|\nabla \times u_q\|_p\\
 & \ \ \ \ \ \ \ \ \ \ \ \ \ \ \ \ \ \ \ \ \ \ \ \ {\color{white}..} \lesssim \lambda_q \|u_q\|_p\sim \|\nabla u_q\|_p,
\end{split}
\end{equation}
for all $1 \leq p \leq \infty$. In particular, there exists $M>0$, such that
\begin{equation}\label{curl-gr}
\lambda_q \|u_q\|_p \leq M \|\nabla \times u_q\|_p, \qquad 1 \leq p \leq \infty.
\end{equation}

Now given $2\leq r<6$ ($2\leq r \leq \infty$ when $b\equiv 0$), let
\begin{equation}\label{parameter-cr}
c_r= \min\{\tilde c, \bar c/M\},
\end{equation}
where $\tilde c$ and $\bar c$ are constants from Propositions~\ref{prop} and \ref{prop2} respectively. We point out that the choice of $c_r$ in (\ref{parameter-cr}) is consistent with the condition $c_r\leq \tilde c$ in Proposition \ref{prop}.

Now we fix
\[
\textstyle s \in \left(\frac{1}{2}, \min\{1, \frac{3}{r}\}\right).
\]
Recall assumption \eqref{criterion} in Theorem~\ref{thm}
\[
\limsup_{q \to \infty} \int_{T/2}^T 1_{q\leq Q(\tau)}\|\Delta_q(\nabla \times u)\|_\infty \, d\tau  \leq c_r.
\]
Since $c_r \leq \bar c/M$, assumption \eqref{criterion2} of Proposition~\ref{prop2} holds thanks to (\ref{curl-gr}).
Applying Gr\"onwall's inequality to (\ref{energy-m}) and employing (\ref{f1-w}) from Proposition \ref{prop2},
we infer
\begin{equation}\label{energy-31}
\begin{split}
&\|u(t)\|_{ H^s}^2+\|b(t)\|_{H^s}^2 \\
\leq &\exp\left(\int_{T/2}^t Cf(\tau) \, d\tau\right)
\left( \|u(T/2)\|_{ H^s}^2+\|b(T/2)\|_{H^s}^2\right) \\
\lesssim & e^{T C M_{q^*}}c_r^{-1}\min\{\nu,\mu\}^{-1} \sup_{\tau \in [T/2,t]}\|u(\tau)\|_{H^s}
\left( \|u(T/2)\|_{ H^s}^2+\|b(T/2)\|_{H^s}^2\right).
\end{split}
\end{equation}
Taking $\sup$ in time in (\ref{energy-31}) yields 
\begin{equation}\label{energy-32}
\sup_{\tau \in [T/2,t]}\|u(\tau)\|_{H^s}\lesssim e^{T C M_{q^*}}c_r^{-1}\min\{\nu,\mu\}^{-1}
\left( \|u(T/2)\|_{ H^s}^2+\|b(T/2)\|_{H^s}^2\right).
\end{equation}
Inserting the bound (\ref{energy-32}) on the right hand side of (\ref{energy-31}), we obtain
\[
\|u(t)\|_{ H^s}^2+\|b(t)\|_{H^s}^2 \lesssim e^{2TCM_{q^*}}c_r^{-1}\min\{\nu,\mu\}^{-2}\left( \|u(T/2)\|_{ H^s}^2+\|b(T/2)\|_{H^s}^2\right)^{2},
\]
for all $t \in (T/2,T)$, which implies that $(u,b)$ is regular on $(0,T]$.

\subsection{Alternative versions of the criterion}
In this subsection we show that regularity conditions stated in Theorem~\ref{thm:compare} are stronger than condition \eqref{criterion}, i.e., Theorem~\ref{thm:compare} is a corollary of Theorem~\ref{thm}.

\begin{Lemma}
Let 
\[
\mathcal{T}_q = \sup\{t \in (T/2,T): Q(\tau) < q, \ \forall \tau \in(T/2,t)\}.
\]
Then $\mathcal{T}_q$ is increasing and
\begin{equation} \label{eq:linT_q_in_a_lemma}
\lim_{q \to \infty} \mathcal{T}_q = T.
\end{equation}
Moreover, either condition (1), (2), (3), or (8) in Theorem \ref{thm:compare} implies regularity criterion (\ref{criterion})
\[
\limsup_{q \to \infty} \int_{T/2}^T 1_{q\leq Q(\tau)}\|\Delta_q(\nabla \times u)\|_\infty \, d\tau  \leq c_r,
\]
 in Theorem \ref{thm}.
\end{Lemma}
\pf
Recall that
\[
\bar Q(t) = \sup_{\tau \in (T/2,t]} Q(\tau),
\]
and hence
\begin{equation} \label{wq:T_q_vs_Q}
\begin{split} 
\mathcal{T}_q &= \sup\{t \in (T/2,T):  Q(\tau) < q, \ \forall \tau \in(T/2,t) \}\\
&= \sup\{t \in (T/2,T): \bar Q(t) \leq q \}.
\end{split}
\end{equation}
Note that $\mathcal{T}_q$ and $\bar Q(t)$ are increasing by definition. In addition, \eqref{Qbar} and the fact that
$(u,b)$ is regular on $(0,T)$ imply that $\bar Q(t)$ is bounded on $[T/2, \tau]$ for every $\tau \in (T/2, T)$, and, consequently, \eqref{eq:linT_q_in_a_lemma} holds.
Indeed, note that $\mathcal T_q$ is the time when $\bar Q(t)$ crosses $q$ due to \eqref{wq:T_q_vs_Q}. So if 
\[\lim_{q\to\infty} \mathcal T_q=T'<T,\]
then $\lim_{t \to T'} \bar Q(t)=\infty$, thanks to the definition of $Q(t)$ and the fact that solution is regular on $[T/2, T']$. That is, $\bar Q(t)$ is unbounded on $[T/2, T']$,  which leads to a contradiction.

\vspace{0.1 in}
\noindent
{\bf Condition (1) implies regularity.} To show that condition (1) in Theorem \ref{thm:compare} implies inequality (\ref{criterion}) in Theorem \ref{thm}, we observe that
\[
\begin{split}
\limsup_{q \to \infty} \int_{T/2}^T 1_{q\leq Q(t)}\|\Delta_q(\nabla\times u)\|_\infty \, dt
&\leq \limsup_{q \to \infty} \int_{T/2}^T 1_{q\leq \bar Q(t)}\|\Delta_q(\nabla\times u)\|_\infty \, dt\\
&= \limsup_{q \to \infty} \int_{\mathcal{T}_q}^T \|\Delta_q(\nabla\times u)\|_\infty \, dt,
\end{split}
\]
where we used \eqref{wq:T_q_vs_Q} to infer the last inequality.

\vspace{0.1 in}
\noindent
{\bf Condition (2) implies regularity.} This follows from the fact that for every $\varepsilon >0$,
\[
\limsup_{q \to \infty} \int_{\mathcal{T}_q}^T \|\Delta_q(\nabla \times u)\|_\infty \, dt \leq  \limsup_{q \to \infty} \int_{T-\varepsilon}^T\|\Delta_q(\nabla \times u)\|_\infty \, dt,
\]
since $\mathcal T_q \to T$ as $q \to \infty$.

\vspace{0.1 in}
\noindent
{\bf Condition (3) implies regularity.}
Now suppose that 
\[
\int_{T/2}^T \sup_{q\leq Q(t)} \|\Delta_q(\nabla \times u)\|_\infty \, dt < \infty,
\]
which implies that
\[
\limsup_{q \to \infty} \int_{\mathcal T_q}^T \sup_{q_1 \leq Q(t)} \|\Delta_{q_1}(\nabla\times u)\|_\infty \, dt =0,
\]
because $\mathcal T_q \to T$ as $q \to \infty$. Therefore, employing the fact that $Q(t)\leq \bar Q(t)$ and (\ref{wq:T_q_vs_Q}), we have
\[
\begin{split}
&\limsup_{q \to \infty} \int_{T/2}^T 1_{q\leq Q(t)}\|\Delta_q(\nabla\times u)\|_\infty \, dt \\
\leq &
\limsup_{q \to \infty} \int_{T/2}^T 1_{q\leq Q(t)} \sup_{q_1 \leq Q(t)} \|\Delta_{q_1}(\nabla\times u)\|_\infty \, dt\\
\leq &\limsup_{q \to \infty} \int_{T/2}^T 1_{q\leq \bar Q(t)} \sup_{q_1 \leq Q(t)} \|\Delta_{q_1}(\nabla\times u)\|_\infty \, dt\\
=& \limsup_{q \to \infty} \int_{\mathcal T_q}^T \sup_{q_1 \leq Q(t)} \|\Delta_{q_1}(\nabla\times u)\|_\infty \, dt\\
 =&0,
\end{split}
\]
where the fact that $\mathcal T_q$ is the time when $\bar Q(t)$ crosses $q$ is used again to obtain the last second equality in the estimate. As a consequence, inequality (\ref{criterion}) in Theorem \ref{thm} is satisfied.

\vspace{0.1 in}
\noindent
{\bf Condition (8) implies regularity.}
Now we show that (8) in Theorem~\ref{thm:compare} implies \eqref{criterion}. First, the regularity of $u(t)$ on $(0,T)$  yields that $u(t)$ is smooth for every $t \in (0,T)$, and hence
\[
\limsup_{q \to \infty} \lambda_q^{-1+\frac3{r}}\|u_q(t)\|_{r}= 0, \qquad \text{for all} \qquad t \in (0,T).
\]
Then due to the weak continuity of $u(t)$ in $L^2(\mathbb R^3)$, reflected in the continuity of $\|u_q(t)\|_2$ on $[0,T]$, the smallness of the jump (8) implies that 
\begin{equation}\label{comp-2}
\begin{split}
&\limsup_{q \to \infty} \lambda_q^{-1+\frac3{r}}\|u_q(T)\|_{r} \\
\leq &\ \limsup_{q \to \infty} \lambda_q^{-1+\frac3{r}}\|u_q(T)-u_q(t)\|_{r}+\limsup_{q \to \infty} \lambda_q^{-1+\frac3{r}}\|u_q(t)\|_{r}\\
< &\ \frac{c_r}{2}\min\{\nu,\mu\},
\end{split}
\end{equation}
provided $t$ is close enough to $T$.
Using the smallness of the jump (8) again and (\ref{comp-2}), we can infer that there exists $\varepsilon>0$, such that
\[
\begin{split}
&\limsup_{q \to \infty} \sup_{t \in [T-\varepsilon, T]} \lambda_q^{-1+\frac3{r}}\|u_q(t)\|_{r} \\
\leq &\  \sup_{t \in [T-\varepsilon, T]} \sup_{q\geq -1 }\lambda_q^{-1+\frac3{r}}\|u_q(t)-u_q(T)\|_{r} 
+\limsup_{q \to \infty} \lambda_q^{-1+\frac3{r}}\|u_q(T)\|_{r}\\
< &\ \frac{c_r}{2}\min\{\nu,\mu\}+\frac{c_r}{2}\min\{\nu,\mu\}= c_r\min\{\nu,\mu\}.
\end{split}
\]
Therefore there exists $q^*$, such that
\[
\lambda_{q}^{-1+\frac3{r}}\|u_{q}(t)\|_{r} < c_r\min\{\nu,\mu\}, \qquad \forall q> q^*,t \in [T-\varepsilon, T],
\] 
which implies that $\Lambda_r(t) \leq 2^{q^{*}}$ for all $t \in [T-\varepsilon, T]$ (see \eqref{wave}), and in fact on $[T/2,T]$ due to the regularity of $u(t)$ on $(0,T)$.
Recall that $\Lambda_r(t) = 2^{Q(t)}$ and hence $Q(t)$ is bounded on $[T/2,T]$ as well. On the other hand, it follows from Bernstein's inequality and the fact that $u\in L^\infty(T/2,T; L^2)$, 
\[ \|u_q\|_\infty \lesssim 2^{3q/2}\|u_q\|_2 \in L^\infty(T/2,T) \ \ \mbox{for every fixed} \ \ q\geq -1.\] 
Therefore, we have
\[\sup_{q\leq Q(t)} \lambda_q \|u_q\|_\infty \in L^\infty(T/2,T).\] 
Consequently, in view of (\ref{equiv-u}), the fact that $Q(t)\leq \bar Q(t)$ and (\ref{wq:T_q_vs_Q}), we infer
\[
\begin{split}
\limsup_{q \to \infty} \int_{T/2}^T 1_{q\leq Q(t)}\|\Delta_q(\nabla \times u)\|_\infty \, dt  &\lesssim
\limsup_{q \to \infty} \int_{T/2}^T 1_{q\leq Q(t)}\lambda_q \|u_q\|_\infty \, dt\\
&= \limsup_{q \to \infty} \int_{T/2}^T 1_{q\leq \bar Q(t)} 1_{q\leq Q(t)}\lambda_q \|u_q\|_\infty \, dt\\
&= \limsup_{q \to \infty} \int_{\mathcal{T}_q}^T 1_{q\leq Q(t)}\lambda_q \|u_q\|_\infty \, dt\\
&\leq \limsup_{q \to \infty} (T-{\mathcal{T}_q})  \sup_{q\leq Q(t)}\lambda_q \|u_q\|_\infty\\
&\lesssim \limsup_{q \to \infty} (T-{\mathcal{T}_q})\\
&=0, 
\end{split}
\]
because $\mathcal T_q \to T$ as $q \to \infty$. Hence \eqref{criterion} holds.

\cbdu

Finally, to establish that either of the conditions (4)--(7) in Theorem \ref{thm:compare} implies inequality (\ref{criterion}) in Theorem \ref{thm} and hence the regularity on $(0,T]$,  we need the following.
\begin{Lemma}\label{compare1}
For all $1\leq l< \infty$ and $1\leq r\leq \infty$,
\begin{equation}\label{comp-3}
\begin{split}
&\limsup_{q \to \infty}\int_{T/2}^T1_{q\leq Q(t)}\lambda_q\|u_q(t)\|_\infty\, dt\\
&\leq (c_r \min\{\nu,\mu\})^{1-l}  \limsup_{q \to \infty}\int_{T/2}^T 1_{q\leq Q(t)}\left(\lambda_q^{-1+\frac2l+\frac{3}{r}}\|u_q(t)\|_{r}\right)^l\, dt\\
&\leq (c_r \min\{\nu,\mu\})^{1-l}  \limsup_{q \to \infty}\int_{\mathcal{T}_q}^T \left(\lambda_q^{-1+\frac2l+\frac{3}{r}}\|u_q(t)\|_{r}\right)^l\, dt.
\end{split}
\end{equation}
\end{Lemma}
\pf
First, we notice that $\lambda_q\leq \Lambda_{r}(t)=2^{Q(t)}$ for $q\leq Q(t)$ and hence
\begin{equation}\label{comp-4}
1\leq \left(\Lambda_{r}(t)/\lambda_q\right)^{2-\frac2{l}}  \ \ \ \ \ \mbox {for}\ \ l\geq 1 \ \ \mbox{and} \ \ q\leq Q(t).
\end{equation}
Concurrently, it follows from (\ref{Q}) that 
\[1\leq (c_r \min\{\nu,\mu\})^{-1}\Lambda_{r}^{\frac3r-1}(t) \|u_Q(t)\|_{r} \]
and hence for $l\geq 1$
\begin{equation}\label{comp-5}
1\leq (c_r \min\{\nu,\mu\})^{-(l-1)}\Lambda_{r}^{\left(\frac3r-1\right)(l-1)}(t) \|u_Q(t)\|_{r}^{l-1}.
\end{equation}
Using Bernstein's inequality, inserting (\ref{comp-4})-(\ref{comp-5}) to the resulted integral, and applying H\"older's inequality, we deduce
\begin{equation}\notag
\begin{split}
&\limsup_{q \to \infty}\int_{T/2}^T 1_{q\leq Q(t)}\lambda_q\|u_q(t)\|_\infty\, dt\\
\leq& \limsup_{q \to \infty}\int_{T/2}^T 1_{q\leq Q(t)}  \lambda_q^{1+\frac3{r}}\|u_q(t)\|_{r}\, dt\\
\leq& \limsup_{q \to \infty}\int_{T/2}^T 1_{q\leq Q(t)}\Lambda_{r}^{2-\frac2l}(t)\lambda_q^{-1+\frac2l+\frac{3}{r}}\|u_q(t)\|_{r}\, dt\\
\leq& (c_r \min\{\nu,\mu\})^{1-l} \limsup_{q \to \infty}\int_{T/2}^T \left(1_{q\leq Q(t)}^{1-\frac1l}\Lambda_{r}^{2-\frac2l+\left(\frac3r-1\right)(l-1)}(t)\|u_Q(t)\|_r^{l-1}\right)\\
&\hspace{1.5in} \cdot \left(1_{q\leq Q(t)}^{\frac1l}\lambda_q^{-1+\frac2l+\frac{3}{r}}\|u_q(t)\|_{r}\right)\, dt\\
\leq &(c_r \min\{\nu,\mu\})^{1-l} \limsup_{q \to \infty}\left(\int_{T/2}^T 1_{q\leq Q(t)} \left(\Lambda_{r}^{-1+\frac2l+\frac3r}(t)\|u_Q(t)\|_{r}\right)^l\, dt\right)^{\frac{l-1}{l}} \\
& \hspace{1.5in} \cdot \left(\int_{T/2}^T 1_{q\leq Q(t)}\left(\lambda_q^{-1+\frac2l+\frac{3}{r}}\|u_q(t)\|_{r}\right)^l\, dt\right)^{\frac{1}{l}} \\
\leq &(c_r \min\{\nu,\mu\})^{1-l}  \limsup_{q \to \infty} \int_{T/2}^T 1_{q\leq Q(t)}\left(\lambda_q^{-1+\frac2l+\frac{3}{r}}\|u_q(t)\|_{r}\right)^l\, dt.
\end{split}
\end{equation}
Thus, the first inequality in (\ref{comp-3}) is justified. In order to see the second inequality in (\ref{comp-3}), we again use the fact that $\mathcal T_q$ is the first time when $\bar Q(t)$ reaches $q$:
\begin{equation}\notag
\begin{split}
& \limsup_{q \to \infty} \int_{T/2}^T 1_{q\leq Q(t)}\left(\lambda_q^{-1+\frac2l+\frac{3}{r}}\|u_q(t)\|_{r}\right)^l\, dt\\
\leq & \limsup_{q \to \infty} \int_{T/2}^T 1_{q\leq \bar Q(t)}\left(\lambda_q^{-1+\frac2l+\frac{3}{r}}\|u_q(t)\|_{r}\right)^l\, dt \\
=& \limsup_{q \to \infty}\int_{\mathcal{T}_q}^T \left(\lambda_q^{-1+\frac2l+\frac{3}{r}}\|u_q(t)\|_{r}\right)^l\, dt.
\end{split}
\end{equation}

\cbdu

As a consequence of Lemma \ref{compare1}, one can see that condition (\ref{criterion}) is satisfied if (4) or (5) in Theorem \ref{thm:compare} holds. On the other hand,  we have for any $\varepsilon>0$
\[
\limsup_{q \to \infty} \int_{\mathcal{T}_q}^T \left(\lambda_q^{-1+\frac2l+\frac{3}{r}}\|u_q(t)\|_{r}\right)^l \, dt \leq  \limsup_{q \to \infty} \int_{T-\varepsilon}^T \left(\lambda_q^{-1+\frac2l+\frac{3}{r}}\|u_q(t)\|_{r}\right)^l \, dt,
\]
since $\mathcal T_q \to T$ as $q \to \infty$, and hence

\[\limsup_{q \to \infty}\int_{\mathcal{T}_q}^T \left(\lambda_q^{-1+\frac2l+\frac{3}{r}}\|u_q(t)\|_{r}\right)^l\, dt\leq 
\lim_{\epsilon\to 0}\limsup_{q \to \infty}\int_{T-\epsilon}^T \left(\lambda_q^{-1+\frac2l+\frac{3}{r}}\|u_q(t)\|_{r}\right)^l\, dt.\]
Therefore, condition (\ref{criterion}) is also satisfied if (6) in Theorem \ref{thm:compare} holds.

In the end, we see that (7) implies (4) in Theorem \ref{thm:compare}, and hence it implies (\ref{criterion}).

\bigskip

\section*{Acknowledgement}

The authors would like to express their deep gratitude for the anonymous referee's valuable suggestions which have improved the manuscript significantly.

\end{document}